\documentclass[11pt]{amsart}
\usepackage{amssymb, latexsym, mathrsfs,  color, lipsum}
\usepackage[normalem]{ulem}
\usepackage[colorlinks=true, pdfstartview=FitV,linkcolor=blue,citecolor=blue,urlcolor=blue]{hyperref}
\usepackage{float}
\usepackage{listings}
\usepackage{changepage}
\usepackage{tablefootnote}
\usepackage{float}
\usepackage{caption}
\usepackage{subcaption}
\definecolor{bg}{rgb}{0.95,0.95,0.95}

\usepackage{listings}
\usepackage[section]{placeins}
\newcommand{\D}{\mathcal{D}}

\newcommand{\p}{\mathfrak{p}}

\advance \topmargin by -\headheight
\advance \topmargin by -\headsep

\evensidemargin \oddsidemargin
\newtheorem {theorem}    {Theorem}[section]

\newtheorem {lemma}      [theorem]    {Lemma}
\newtheorem {corollary}  [theorem]    {Corollary}
\newtheorem {proposition}[theorem]    {Proposition}

\theoremstyle{definition}
\newtheorem{definition}[theorem]{Definition}
\newtheorem{remark}[theorem]{Remark}
\newtheorem{ex}[theorem]{Example}

\numberwithin{equation}{section}

\usepackage[colorinlistoftodos]{todonotes}

\newenvironment{red}{\relax\color{red}}{\relax}
\newenvironment{blue}{\relax\color{blue}}{\hspace*{.5ex}\relax}

\newcommand{\ber}{\begin{red}}
	\newcommand{\er}{\end{red}}
\newcommand{\beb}{\begin{blue}}
	\newcommand{\eb}{\end{blue}}

	\title[Machine learning class numbers of real quadratic fields]{
		Machine learning class numbers of real quadratic fields}
	
	\date{\today}
	
	\author[M. Amir]{Malik Amir}
	\address{Microsoft Research New England, One Memorial Drive, Cambridge, MA 02140, U.S.A.}
	\email{malik.amir.math@gmail.com}
	
	\author[Y.-H. He]{Yang-Hui He}
	\address{London Institute for Mathematical Sciences, Royal Institution, London W1S 4BS, U.K. \newline \indent Department of Mathematics, City, University of London, EC1V 0HB, U.K. \newline \indent
		Merton College, University of Oxford, OX14JD, U.K. \newline \indent
		School of Physics, NanKai University, Tianjin, 300071, P.R.~China}
	\email{hey@maths.ox.ac.uk}
	
	\author[K.-H. Lee]{Kyu-Hwan Lee}
	\address{Kyu-Hwan Lee, Department of Mathematics, University of Connecticut, Storrs, CT 06269, U.S.A.}
	\email{khlee@math.uconn.edu}
	
	\author[T. Oliver]{Thomas Oliver}
	\address{Thomas Oliver, Teesside University, Middlesbrough, U.K.}
	\email{T.Oliver@tees.ac.uk}
	
	\author[E. Sultanow]{Eldar Sultanow}
	\address{Eldar Sultanow, Capgemini Germany, Bahnhofstraße 30, 90402 Nuremberg, Germany}
	\email{eldar.sultanow@capgemini.com}
	
	\subjclass[2020]{11R29, 11R80, 62R99}
	
\begin{document}
		
	\begin{abstract}  
		We implement and interpret various supervised learning experiments involving real quadratic fields with class numbers 1, 2 and 3. 
		We quantify the relative difficulties in separating class numbers of matching/different parity from a data-scientific perspective, 
		apply the methodology of feature analysis and principal component analysis, 
		and use symbolic classification to develop machine-learned formulas for class numbers 1, 2 and 3 that apply to our dataset.
	\end{abstract}
	
	\maketitle
	\tableofcontents

	\section{Introduction}
	The class number of a real quadratic field, or, more generally, of a number field, measures how far its ring of integers is from being a unique factorization domain (UFD), with class number 1 meaning that the ring is a UFD. 
	The {\em Gauss class number problem} for real quadratic fields concerns whether or not there are infinitely many real quadratic fields with class number 1. 
	This fundamental question remains one of the central open questions in number theory to this day.  
	
	In a recent paper \cite{HLOb}, it was observed that supervised learning techniques could be used to distinguish real quadratic fields of class number 1 from those of class number 2.
	That article was the second in a series, featuring also \cite{HLOa} and \cite{HLOc}, in which the unifying theme was the application of machine learning algorithms to arithmetic objects presented by finite lists of coefficients in certain Dirichlet series called $\zeta$-functions (or $L$-functions).
	Subsequently, this series of articles has been enhanced by \cite{HLOP}, which documents a first glimpse of unexpected phenomena amongst the coefficients.
	In all of these papers, the methodology is motivated by the general philosophy that arithmetic objects may be classified through these functions; in particular, various interesting arithmetic invariants appear in their Laurent expansions.
	
	In this article, we seek to formalise, generalise, and interpret the aforementioned supervised learning experiment for real quadratic fields.
	We begin by reviewing some mathematical theory, including the {\em genus field} which leads to a strategy for distinguishing between real quadratic fields of class numbers 1 and 2 using finitely many $\zeta$-coefficients (i.e. coefficients of $\zeta$-functions) and no other invariants.
	Subsequently, we will observe that certain machine learning techniques implicitly recognize these consequences of genus theory in their classifiers. 
	
	
	Strategies derived from genus theory break down when considering class numbers 1 and 3.
	A data-scientific perspective will allow us to express a sense in which separating class number 1 from class number 3 is truly more challenging, without reference to genus theory.
	Consequently, applying standard machine learning classifiers to the $\zeta$-coefficients of real quadratic fields with class numbers 1 and 3 does not yield high accuracy predictors.
	To rectify this, we incorporate some, but not necessarily all, additional features inspired by the analytic class number formula for real quadratic fields, that is: 
	\begin{equation}\label{eq.ACNF} 
	\lim_{s \rightarrow 1} (s-1) \zeta_d(s) = \frac {2 R_d h_d}{ \sqrt{D}},
	\end{equation} 
	where $\zeta_d(s)$ is the Dedekind zeta function of $\mathbb{Q}(\sqrt{d})$, $D$ is the discriminant, $R_d$ is the regulator, and  $h_d$ is the class number.
	
	In the rest of this introduction, we overview the subsequent sections.
	In Section~\ref{s:theory}, we review Dedekind $\zeta$-functions and establish various results in genus theory that explain several statistical observations about the class numbers of real quadratic fields.
	In Section~\ref{ss.sep}, we explore the pairwise separation of real quadratic fields with different class numbers using finitely many coefficients of their Dedekind $\zeta$-functions.
	More precisely, we introduce a cost function which quantifies the separability, and may be heuristically computed and optimised using the so-called \textit{bubble algorithm}.
	The cost function is built from certain counting functions, which enumerate square-free $d$ such that $\mathbb{Q}(\sqrt{d})$ has specified ramification properties.
	Section~\ref{ss.sep} is complemented by Appendix~\ref{app.dimred}, in which we explore separation of $\zeta$-coefficient data using principal component analysis (PCA). 
	
	In Section~\ref{s:12}, we investigate the binary classification of real quadratic fields with class numbers 1 and 2 using gradient boosting tree based learning algorithms (specifically, LightGBM and CatBoost) and genetic programming (specifically, symbolic classification).
	In particular, we undertake a supervised learning experiment in the style of \cite{HLOb}, that is, using finite lists of $\zeta$-coefficients as features, but place greater emphasis on alternative methodologies and feature analysis.
	Furthermore, rediscover some results from genus theory, first stated in Section~\ref{s:theory}, concerning the parity of class numbers.
	Section~\ref{s:12} is complemented by Appendix~\ref{s:app}, in which we record additional plots and metrics for the experiments.
	
	In Section~\ref{s:13}, we investigate the binary classification of class numbers 1 and 3, using LightGBM. 
	Having previously observed that $\zeta$-coefficients are not sufficient for high accuracy classifiers in this case,
	we also incorporate various combinations of features consisting connected to equation~\eqref{eq.ACNF}.
	In particular, we will involve the ramified primes, and
	some partial sums related to the Dedekind zeta function and the Dirichlet $L$-function.
	Applying the symbolic classifier, we are lead to two approximate formulas for the class number, the first of which is essentially equation~\eqref{eq.ACNF}, and the second of which looks somewhat different but nevertheless suggests that $h_d$ is proportional (resp. inversely proportional) to $\sqrt{D}$ (resp. $R_d$). 
	
	\subsection*{Acknowledgements}
	MA is supported by Microsoft Research NE, YHH is indebted to STFC UK, for grant ST/J00037X/2, 
	KHL is partially supported by a grant from the Simons Foundation (\#712100), 
	and TO acknowledges support from the EPSRC through research grant EP/S032460/1. MA is also grateful to Mounir Boukadoum, Henry Cohn, Vlad Serban and Maryna Viazovska for their valuable comments, discussions and teachings.
	
	\section{Genus theory}\label{s:theory}
	In \cite{HLOb}, we saw that a random forest classifier trained with a dataset of coefficients for the Dedekind zeta function was able to distinguish between real quadratic fields of class numbers 1 and 2 to high accuracy.
	In order to explain mathematically what a machine learns from the dataset, we review the genus field of a real quadratic field and establish various constraints on its class number in terms of the number of ramified primes. 
	In Section~\ref{ss.sep}, we will quantify one way in which the analogous learning task for real quadratic fields with class numbers $1$ and $3$ is more challenging.
	In Section~\ref{s:12}, we will use the results in this section to inform the feature analysis undertaken.

	\subsection{Dedekind zeta functions}\label{s:DZF}
	Let $K$ be a number field with ring of integers $\mathcal{O}_K$. 
	The Dedekind zeta function of $K$ is defined to be:
	\begin{equation}\label{eq.DedZet}
	\zeta_K(s)=\sum_{I\leq\mathcal{O}_K}N(I)^{-s}=\prod_{\mathfrak{p}\leq\mathcal{O}_K}(1-N(\mathfrak{p})^{-s})^{-1}, \ \ \mathrm{Re}(s)>1,
	\end{equation}
	in which $N$ denotes the norm map and the sum (resp. product) is over the non-zero (resp. prime) ideals of $\mathcal{O}_K$.
	
	For a square-free integer $d$, let $K_d=\mathbb{Q}(\sqrt{d})$.
	We say that $K_d$ is real (resp. imaginary) if $d>0$ (resp. $d<0$).
	The discriminant $D$ of $K_d$ is given by
	\begin{equation}\label{eq.QuadDisc}
	D=\begin{cases}d& \text{ if } d\equiv1~(\mathrm{mod}~4),\\4d& \text{ if } d\equiv2,3~(\mathrm{mod}~4).\end{cases}
	\end{equation}
	When $K=K_d$, we will write $\zeta_K(s)=\zeta_d(s)$. 
	We note that
	\begin{equation}\label{eq.zetafactorization}
	\zeta_d(s)=\zeta(s)L(s,\chi_D),
	\end{equation}
	where  $\chi_D:= \left ( \frac D \cdot \right )$ is the Kronecker symbol attached to $K_d$, $L(s,\chi_D)$ is the associated Dirichlet $L$-function, and $\zeta(s)$ is the Riemann zeta function. 
	By $h_d$, we denote the class number of $K_d$. 
	
	We may write \[ \zeta_d(s)=\sum_{n=1}^{\infty}a_nn^{-s},\] in which $a_n$ denotes the $n$th Dirichlet coefficient of $\zeta_d(s)$.
	We will refer to the sequence $(a_n)_{n=1}^{\infty}$ as the $\zeta$-coefficients of $K_d$.
	Using equation~\eqref{eq.zetafactorization}, we deduce that
	\begin{equation}\label{eq.an}
	a_n=\sum_{m|n}\chi_D(m),
	\end{equation}
	where the sum is over $m$ dividing $n$.
	In particular, if $n=p$ is prime, then equation~\eqref{eq.an} simplifies to 
	\begin{equation}\label{eq.ap}
	a_p=1+\chi_D(p).
	\end{equation}
	Since $\chi_D$ is a real quadratic character, we have $\chi_D(p)\in\{-1,0,1\}$.
	Thus, for $p$ prime, equation~\eqref{eq.ap} implies that $a_p\in\{0,1,2\}$.
	By construction, we have $\chi_D(p)=0$ if and only if $p$ divides $D$.
	Subsequently, it follows from equation~\eqref{eq.ap} that 
	\begin{equation} \label{eq-div}
	\text{$a_p=1$  \quad $\Longleftrightarrow$ \quad $p$ divides $D$  \quad $\Longleftrightarrow$ \quad  $p$ is ramified in $K_d$.} \end{equation}

	\subsection{Genus fields}\label{genus}
	
	By definition, the {\em genus field} $E_d$ of $K_d$ is the maximal unramified extension of $K_d$ which is abelian over $\mathbb{Q}$, and the {\em extended} genus field $E^+_d$ of $K_d$ is the maximal extension of $K_d$ which is unramified at all finite primes and abelian over $\mathbb{Q}$. 
	Recall that a prime discriminant is a discriminant divisible by a single prime. 
	Write $D=d_1 \cdots d_t$ as a product of prime discriminants $d_i$ with $t=\omega(D)$, the number of distinct primes dividing $D$. 
	
	We will use the following results on genus fields (see, e.g., \cite{Ja,Lemm}): 
	\begin{proposition} We have
		\begin{align} \label{explicitGd} E^+_d = \mathbb Q(\sqrt{d_1}, \dots , \sqrt{d_t}), && E_d= E_d^+ \cap \mathbb R , \\
			\label{gdkd}
			\mathrm{Gal}(E_d^+/K_d)\cong C^+_d/2C^+_d \cong (\mathbb Z/ 2 \mathbb Z)^{t-1},& & \mathrm{Gal}(E_d/K_d)\cong C_d/2C_d \cong (\mathbb Z/ 2 \mathbb Z)^s,
		\end{align}
		where $C^+_d$ denotes the narrow class group of $K_d$  and $C_d$ the class group of $K_d$, and $s=t-1$ if $d$ is a sum of two squares and $s=t-2$ otherwise.
	\end{proposition}

	Let $n_d$ denote the number of rational primes ramified in $K_d$. Since the ramified primes in $K_d$ are precisely those which divide $D$, we have $n_d=t=\omega(D)$. The above proposition implies a general result given below on the parity of class numbers $h_d$.
	
	\begin{corollary} \label{iff}
		The class number $h_d$ is odd if and only if either $n_d=1$ or $D=d_1d_2$ with prime discriminants $d_1, d_2 <0$.
	\end{corollary}
	
	\begin{proof}
		Assume that $h_d$ is odd. Then $s=0$ in \eqref{gdkd} and so $t=1$ or $t=2$. If $t=n_d=1$ then we are done. If $t=2$ then $d$ cannot be a sum of two squares. Then $d=p_1$, $d=2p_1$ or $d=p_1p_2$ with primes $p_1, p_2 \equiv 3 \pmod 4$. In all three cases, $D$ is a product of two negative prime discriminants. Conversely, if $n_d=t=1$ then $s=0$ and $h_d$ is odd from \eqref{gdkd}, and if $D=d_1 d_2$ with $d_1, d_2 <0$, then $t=2$, $s=t-2=0$ and $h_d$ is odd.  
	\end{proof}

	Now we establish a series of lemmas which can be observed in the dataset of $\zeta$-coefficients. These lemmas will be used for the interpretation of machine learning classifiers in the next sections.
	
	\begin{lemma}\label{l:1rp}
		If $n_d=1$ then $h_d \neq 2$. In particular, if $n_d=1$ and $h_d\leq2$, then $h_d=1$.
	\end{lemma}
	
	For example, if $d=229$, then $n_d=1$ and $h_d=3$.
	
	\begin{proof}
		If $K_d$ has only one ramified prime, then either $d=2$ or $d=p$ for an odd prime $p\equiv1$ mod $4$.
		If $d=2$ then the class number is $1$.
		It remains to consider the case that $d=p$ is a prime congruent to $1$ mod $4$.
		In this case, by \eqref{explicitGd}, we have $E_d=E^+_d=K_d$ and hence $\mathrm{Gal}(E_d/K_d)$ is trivial. 
		If $C_d$ has order $2$, then equation~\eqref{gdkd} implies that $\mathrm{Gal}(E_d/K_d)$ would also have order $2$, which is a contradiction.
	\end{proof}
	
	\begin{lemma}\label{l:p14}
		If $h_d=1$ and $a_p=1$ for some prime $p\equiv1$ mod $4$, then $d=p$.
	\end{lemma}
	
	We cannot allow primes $p\equiv3$ mod $4$ in Lemma~\ref{l:p14}.
	For example, $h_{14}=1$ and $a_7=1$ for $K_{14}$. 
	
	\begin{proof}
		Since $h_d=1$, we have $E_d=E_d^+ \cap \mathbb R=K_d$. Since $a_p=1$, we know that $p$ divides $D$, and $p$ is a prime discriminant. It follows from \eqref{explicitGd} that $\sqrt p \in E_d^+ \cap \mathbb R=K_d$. Then we have $\mathbb Q(\sqrt p) = K_d$.  
	\end{proof}
	
	\begin{lemma}\label{l:even}
		If $h_p=1$ for a prime $p\equiv1$ mod $4$ then $h_{mp}\geq2$ for square-free integers $m\in\mathbb{Z}_{>1}$ not divisible by $p$. 
	\end{lemma}
	
	We cannot allow primes $p\equiv3$ mod $4$ in Lemma~\ref{l:even}. 
	For example, $h_7=h_{14}=1$.
	
	\begin{proof}
		Given square-free $m\in\mathbb{Z}_{>1}$ and $p\equiv1$ mod $4$, the field $K_{mp}$ satisfies $a_p=1$.
		Since $mp\neq p$, Lemma~\ref{l:p14} already implies that $h_{mp}\neq1$.
	\end{proof}
	
	\begin{remark}
		For $m\geq1$, the field $K_{mp}$ satisfies $a_p=1$ by \eqref{eq-div}.
		If $p\equiv1$ mod $4$ and $h_p=1$, then Lemma~\ref{l:even} implies the existence of infinitely many real quadratic fields with 
		class number $\geq2$ such that $a_p=1$. 
		This will be relevant in Example~\ref{ex.a2a5a7}.
	\end{remark}
	
	\begin{lemma}\label{l:3rp}
		If $n_d\geq3$, then $h_d \ge 2$.
		In particular, if $n_d \ge 3$ and $h_d\leq2$, then $h_d=2$.
		Similarly, if $n_d \ge 4$, then $h_d \ge 4$.
	\end{lemma}
	
	Note that it is possible that $K_d$ satisfies $n_d=2$ and $h_d=1$.
	This occurs, for example, if $d=7$, in which case $h_d=1$ and $K_d$ is ramified at the primes $2$ and $7$. 
	
	\begin{proof}
		The assertions follow from \eqref{gdkd}. 
	\end{proof}
	
	\begin{lemma}\label{l:srp}
		Assume that $n_d=2$ and $h_d\in\{1,2\}$.
		Let $p_1$ be the smallest ramified prime in $K_d$.
		If $p_1\equiv1$ mod $4$ (resp. $p_1\equiv3$ mod 4), then $h_d=2$ (resp. $h_d=1$).
	\end{lemma}
	
	\begin{proof}
		Write $D=d_1d_2$ with $d_i$ prime discriminants. If $p_1=d_1 \equiv 1$ mod $4$ then $d_2$ is also a prime $\equiv 1$ mod $4$. Further, if $h_d=1$, then Lemma~\ref{l:p14} yields a contradiction. Thus $h_d=2$.
		
		If $p_1= - d_1 \equiv 3$ mod $4$ then $d_2 = - p_2$ for a prime $p_2 \equiv 3$ mod $4$. We obtain from \eqref{explicitGd} that $E^+_d= \mathbb Q( \sqrt{-p_1}, \sqrt{-p_2})$ and $E_d=\mathbb Q(\sqrt{p_1p_2}) =K_d$. By \eqref{gdkd}, we have $h_d=1$.   
		
	\end{proof}

	\section{Separability of $\zeta$-coefficients}\label{ss.sep}
	
	In this section, we explore the relative difficulty in distinguishing between different pairs of class numbers using only the associated $\zeta$-coefficients.
	This is in keeping with the strategy implemented in \cite{HLOb}.
	In Sections~\ref{s:12} and~\ref{s:13}, we will go beyond \cite{HLOb} and incorporate other features inspired by the analytic class number formula into our dataset.
	
	The separability of $\zeta$-coefficients will be quantified in terms of a cost function, which will be introduced in Section~\ref{s:ba}. 
	The cost function is constructed in terms of certain counting functions explored in Section~\ref{ss:cf}.
	The counting functions and the cost function may be heuristically computed and optimised using the so-called \textit{bubble algorithm} introduced in Section~\ref{s:visualisation}.
	In Section~\ref{s:ba13}, we will use the bubble algorithm to establish that the problem of distinguishing between class numbers 1 and 3 is more challenging than between class numbers 1 and 2.
	In Section~\ref{s:ci}, we will explore the role played by $\zeta$-coefficients with non-prime indices in the separation of $\zeta$-coefficients.
	In Appendix~\ref{app.dimred}, we document alternative approaches to the  separation of $\zeta$-coefficients using PCA. 
	
	\subsection{Counting functions}\label{ss:cf}
	
	\begin{definition}
		For a positive real number $X$,
		a positive integer $h\in\mathbb{Z}_{>0}$, 
		a vector of non-negative integers $v\in\mathbb{Z}_{\geq0}^3$, 
		and positive integers $\ell,m,n\in\mathbb{Z}_{>0}$, 
		let $f_h^v(\ell,m,n)$ be the number of square-free positive integers $d<X$ such that $h_d=h$ and $(a_{\ell},a_m,a_n)=v$,
		where $a_{\ell}$ (resp. $a_m$, $a_n$) is the $\ell$th (resp. $m$th, $n$th) coefficient of $\zeta_d(s)$.
	\end{definition}
	
	In symbols, we have
	\begin{equation}\label{eq.fhv}
	f_h^v:\mathbb{Z}_{>0}^3\rightarrow\mathbb{Z}_{\geq0}, \ \ f_h^v(\ell,m,n)=\#\{d<X:h_d=h \text{ and } (a_{\ell},a_m,a_n)=v\}.
	\end{equation}
	The function $f_h^v$ depends on $X$, though this fact is suppressed from the notation.
	
	For rational primes $p,q,r$, if $f_h^v(p,q,r)>0$ then there exists a real quadratic field with class number $h$ and certain ramification at $p,q,r$ prescribed by the vector $v$.
	Indeed, we have already mentioned in Section~\ref{s:DZF} that we have $a_p,a_q,a_r\in\{0,1,2\}$ and a prime $p$ ramifies in $K_d$ if and only if the corresponding coefficient $a_p$ of $\zeta_d(s)$ is $1$.
	Furthermore, we know that a prime is split (resp. inert) if the corresponding coefficient is $2$ (resp. $0$) (cf. \cite[equation~(4.7)]{Ko}). 
	
	We will refer to the inputs of $f_h^v$ as triples, and use the notation $[\ell,m,n]$ so as to distinguish the inputs from the indexing vectors $v$.
	
	\begin{ex}\label{ex.a2a5a7}
		Let $h=1$ and $[\ell,m,n]=[3,5,7]$. Then, for $v=(0,1,0)$ and $X>5$,  
		we have $f_1^v(3,5,7)=1$.
		For all the other $v=(a,1,c)$ for $a,c\in\{0,1,2\}$ and $(a,c) \neq (0,0)$, we have $f_1^v(3,5,7) =0$. Indeed, if $a_5=1$, then $K_d$ is ramified at $5$, and so $d=5m$ for square-free $m$ not divisible by $5$.
		Since $h_5=1$ and $5\equiv1$ mod $4$, Lemma~\ref{l:even} implies that $h_{5m}\geq2$ for $m >1$.
		When $d=5$, we have $a_3=a_7=0$, proving the claim. \end{ex}
	
	\begin{definition}
		For $h\in\mathbb{Z}_{>0}$ and a triple $[\ell,m,n]$, let $g_h(\ell, m,n)$ be the number of $v=(v_1,v_2,v_3) \in \mathbb Z_{\ge 0}^3$ such that $f_h^v(\ell, m,n)>0$, where $f_h^v$ is as in equation~\eqref{eq.fhv}. In symbols, we have
		\begin{equation}\label{gh}
		g_h:\mathbb{Z}^3_{>0}\rightarrow\mathbb{Z}_{\geq0}, \ \ g_h(\ell,m,n)=\#\{v=(v_1,v_2,v_3)\in\mathbb{Z}_{\geq0}^3:f_h^v(\ell,m,n)>0 \}.
		\end{equation}
	\end{definition}
	
	In order to gain some familiarity with these functions, we consider first the case that $\ell,m,n$ are all primes (composites will appear in the sequel).
	Then, since $a_p,a_q,a_r\in\{0,1,2\}$ for primes $p,q,r$, a trivial bound is given by the number of all possible vectors $v$, i.e.,   
	\begin{equation}\label{eq.ghpqr}
	g_h(p,q,r)\leq 3^3=27, \ \  \left(h \in \mathbb Z_{>0}\right). \end{equation}
	
	\begin{ex}\label{gh357} 
		For $h\in\{1,2,3\}$ and $[\ell,m,n]=[3,5,7]$, we will calculate $g_h(3,5,7)$. In Table~\ref{tab:g123}, we list the smallest $d$ such that $K_d$ has class number given by row and $(a_3,a_5,a_7)$ given by column. An entry $\times$ indicates that the vector does not occur, which may be verified using the statements in Section~\ref{genus}. Indeed, if $a_5=1$ then it follows from Corollary \ref{iff} that $h_d$ is odd only when $d=5$. This accounts for 17 occurrences of $\times$ in Table~\ref{tab:g123}. In all the remaining three occurrences, we have $a_3=a_7=1$. Again by Corollary \ref{iff}, we have $h_d$ is odd only when $d=21$, which explains exactly the three occurrences of $\times$.  
		
		\begin{table}[ht!]
			\begin{center}
				\caption{\sf Smallest $d$ such that $K_d$ has class number given by row and $(a_3,a_5,a_7)$ given by column. An entry $\times$ indicates that the vector does not occur.}\label{tab:g123}
				\begin{tabular}{|c|c|c|c|c|c|c|c|c|c|} 
					\hline
					$h_d$ &  $(0,0,0)$ &$(0,0,1)$ &$(0,0,2)$ &$(0,1,0)$ &$(0,1,1)$ &$(0,1,2)$ &$(0,2,0)$ &$(0,2,1)$ &$(0,2,2)$  \\ \hline  1&17&77&19&5&$\times$&$\times$& 41&14&11 \\
					2&122&182&218&185&35&65&26&119&74\\
					3&257&2177&473&$\times$&$\times$&$\times$&761&2429&254\\ \hline
					$h_d$ &$(1,0,0)$ &$(1,0,1)$ &$(1,0,2)$ &$(1,1,0)$ &$(1,1,1)$&$(1,1,2)$&$(1,2,0)$&$(1,2,1)$ &$(1,2,2)$\\
					\hline  1& 3&$\times$&57&$\times$&$\times$&$\times$&6&21&141\\
					2&87&42&78&285&105&15&66&609&39\\
					3&993&$\times$&1257&$\times$&$\times$&$\times$&321&$\times$&1101\\ \hline
					$h_d$ &$(2,0,0)$ &$(2,0,1)$ &$(2,0,2)$ &$(2,1,0)$ &$(2,1,1)$&$(2,1,2)$&$(2,2,0)$&$(2,2,1)$ &$(2,2,2)$\\
					\hline  1& 13&7&22&$\times$&$\times$&$\times$&19&301&46\\
					2&178&238&58&10&70&85&34&91&106\\
					3&733&7273&142&$\times$&$\times$&$\times$&229&469&316\\ \hline
				\end{tabular}
			\end{center}
		\end{table}
		Using Table~\ref{tab:g123}, we deduce:
		\begin{equation}\label{eq.g1g2357}
		g_h(3,5,7)=\begin{cases}18,& h=1,~X>301,\\27,&h=2,~X>609,\\16,&h=3,~X>7273.\end{cases}
		\end{equation}
		In particular, we see that there exists a real quadratic field of class number $2$ for every possible combination of ramification at $3,5,7$.
	\end{ex}
	
	\begin{definition}
		For distinct positive integers $i,j\in\mathbb{Z}_{>0}$, let $g_{i,j}$ count the number of $v$ so that both $f_{i}^v$ and $f_{j}^v$ are positive. That is, we define 
		\begin{equation}\label{gij}
		g_{i,j}(\ell,m,n)=\#\{v\in\mathbb{Z}_{\geq0}^3:f_{i}^v(\ell,m,n)>0\text{ and }f_{j}^v(\ell,m,n)>0\}.
		\end{equation}
		Clearly, we have $g_{i,j}\leq\min\{g_i,g_j\}$.
	\end{definition}
	
	\begin{ex}\label{exg12g13}
		Using Table~\ref{tab:g123}, we obtain:
		\begin{equation}
		g_{1,j}(3,5,7)=\begin{cases}18,& j=2,~X>609,\\27,&j=3,~X>7273.\end{cases}
		\end{equation}
	\end{ex}
	
	\subsection{Composite indices}\label{s:ci}
	In our initial investigation of the counting functions introduced in Section~\ref{ss:cf}, we considered only prime indices. With composite indices, the coefficients $a_n$ can take many more values than their prime counterparts.
	Indeed, each character value appearing in equation~\eqref{eq.an} is in $\{-1,0,1\}$, 
	and so we see that $a_n$ is an integer bounded by $1\pm\Omega(n)$,
	where $\Omega(n)$ denotes the number of prime factors dividing $n$ (counted with multiplicity).
	On the other hand, the coefficient $a_n$ counts the number of ideals of norm $n$, which is a non-negative integer.
	Combining these observations, we deduce
	\begin{equation}\label{anrangerefined}
	a_n\in\{0,1,\dots,\Omega(n)-1,\Omega(n),\Omega(n)+1\}.
	\end{equation} 
	The number of values actually achieved by $a_n$ depends on the multiplicity in the prime factorisation of $n$.
	
	\begin{ex}
		If $n=p^2$ for some prime $p$, then, equation~\eqref{eq.zetafactorization} implies that 
		\begin{equation}\label{eq.a9}
		a_{p^2}=1+\chi_D(p)+\chi_D(p)^2\in\{1,3\}.
		\end{equation}
		In equation~\eqref{eq.a9}, we have $a_{p^2}=3$ (resp. $a_{p^2}=1$) if and only if $p$ is ramified (resp. unramified).
		On the other hand, if $n=pq$ is a product of two distinct primes, then equation~\eqref{eq.zetafactorization} implies that
		\begin{equation}\label{eq.a6}
		a_{pq}=1+\chi_D(p)+\chi_D(q)+\chi_D(p)\chi_D(q)\in\{0,1,2,4\}.
		\end{equation} 
		In equation~\eqref{eq.a6}, we have $a_{pq}=0$ if and only if $p$ or $q$ is inert, $a_{pq}=1$ if and only if $p$ and $q$ are ramified, $a_{pq}=2$ if and only if $K_d$ is ramified at one of $p$ and $q$ and splits at the other, and $a_{pq}=4$ if and only if $p$ and $q$ split.
	\end{ex}
	
	Using equation~\eqref{anrangerefined}, we deduce the following analogue of equation~\eqref{eq.ghpqr} for non-prime indices:
	\begin{equation}\label{eq.ghcomp}
	g_h(\ell,m,n)\leq(\Omega(\ell)+2)(\Omega(m)+2)(\Omega(n)+2).
	\end{equation}

	\subsection{Cost function}\label{s:ba}
	In order to measure separability of class numbers $1,2,3$ in our datasets of $\zeta$-coefficients, we will introduce a cost function which may be heuristically calculated using the searching algorithm described in Section~\ref{s:visualisation}. 
	
	The cost function is constructed so as to account for two key considerations.
	On one hand, the cost function favours triples which minimise $g_{i,j}$; that is, we are interested in $\mathrm{argmin}(g_{i,j})\subset\mathbb{Z}_{>0}^3$.
	This is natural, since $g_{i,j}$ is a coarse measure of the extent to which the sets $\{K_d:h_d=i\}$ and $\{K_d:h_d=j\}$ may be separated by the associated $\zeta$-coefficients.
	Indeed, if $g_{i,j}$ were to hypothetically take the value $0$ at some triple $[\ell,m,n]$, then each $v\in\mathbb{Z}^3_{\geq0}$ would correspond to at most one class number.
	
	On the other hand, the minimisation of $g_{i,j}$ needs to be taken in a relative way. Namely, if the union $\{v : f^v_i(\ell,m,n)>0\}\cup\{v:f_j^v(\ell,m,n) >0\}$ is small, the size $g_{i,j}$ of intersection would also tend to be small.  
	Consequently, we define our cost function $C_{i,j}$ to be the ratio of the intersection over the symmetric difference:
	\begin{equation}\label{eq.costfunction}
	C_{i,j}(\ell,m,n):=\frac{g_{i,j}(\ell,m,n)}{g_i(\ell,m,n)+g_j(\ell,m,n)-2g_{i,j}(\ell,m,n)}
	\end{equation}
	for any pair of class numbers $\{i,j\}$ and any triple $[\ell,m,n] \in \mathbb Z_{>0}^3$.
	Then we are searching for
	\begin{equation}\label{eq.Uij}
	\mathrm{argmin}(C_{i,j}).
	\end{equation}
	We will find heuristic solutions to equation~\eqref{eq.Uij} using a searching algorithm described in the next section.
	\begin{remark}\label{rem.LMFDB}
		In our heuristic calculations of the counting functions and cost function, we will count only $d$ whose discriminant $D$ appears in the LMFDB.
		Though the LMFDB is complete for $D<2\times10^6$, it includes some larger $d$ and, for the purposes of this section, we note that the largest $d$ such that $h_d=1$ (resp. $h_d=2$, resp. $h_d=3$) is $34,554,953$ (resp. $43,723,857$, resp. $35,598,713$). The number of real quadratic fields in our dataset for each class number is given in Table \ref{tab:summary}.
		
		\begin{table}[ht!]
			\begin{center}
				\caption{\sf Number of real quadratic fields in our dataset of real quadratic fields extracted from the LMFDB.} 
				\label{tab:summary}
				\begin{tabular}{c|c|c} 
					$\#\{h_d=1\}$&$\#\{h_d=2\}$&$\#\{h_d=3\}$\\
					\hline
					177159&183436& 25201 
				\end{tabular}
			\end{center}
			
		\end{table}

	\end{remark}
	
	\begin{ex}\label{e:=0}
		From the dataset available in the LMFDB, we find 
		\begin{equation}\label{eq.g12}
		g_{1,2}(3,5,7)=g_{1,2}(665,740,940)=g_{1,2}(520,783,991)=18.
		\end{equation}
		The value $18$ is much lower than the tentative upper bound which comes from equation~\eqref{eq.ghcomp}. 
		Indeed, we have $\Omega(3)=\Omega(5)=\Omega(7)=1$,  $\Omega(665)=\Omega(740)=\Omega(940)=3$, $\Omega(520)=5$, $\Omega(783)=4$, $\Omega(991)=1$, 
		and so equation~\eqref{eq.ghcomp} implies that
		\begin{equation}\label{ghcompbound}
		g_h(3,5,7)\leq27, \ \ g_h(665,740,940)\leq125, \ \ g_h(520,783,991)\leq126,
		\end{equation}
		for any positive integer $h$.
		The discrepancy between the value given in equation~\eqref{eq.g12} and the hypothetical upper bound in equation~\eqref{ghcompbound} means that there could potentially be many triples $[\ell,m,n] \in \mathbb{Z}^3_{\geq0}$ that make $g_i(\ell,m,n) -g_{i,j}(\ell,m,n)$ and $g_j(\ell,m,n)-g_{i,j}(\ell,m,n)$ large and $C_{i,j}(\ell,m,n)$ small.
	\end{ex}
	
	\subsection{The bubble algorithm for $h_d\in\{1,2\}$}\label{s:visualisation}
	In order to find a solution to equation~\eqref{eq.Uij}, we utilise a searching algorithm referred to as the \textsl{bubble algorithm}. 
	More generally, the bubble algorithm may be used to evaluate all counting functions introduced so far.
	In this section, we will focus on $\{i,j\}=\{1,2\}$.
	The generalisation to other pairs of class numbers is straightforward.
	
	The terminology ``bubble'' is motivated by certain visualisations (bubble charts) of the value distributions of $\zeta$-coefficients, such as Figure~\ref{fig:bubble_chart}.
	In Figure~\ref{fig:bubble_chart} left (resp. right), we see a cube with axes given by coefficient triples $v=(a_{3},a_5,a_7)$ (resp. $v=(a_2,a_3,a_5)$). 
	At each integer vector $v$, we see a coloured bubble whose size is determined by $f^v_1(\ell,m,n)+f^v_2(\ell,m,n)$.
	The presence of the colour red (resp. green) at $v$ indicates that $f_1^v(\ell,m,n)>0$ (resp. $f_2^v(\ell,m,n)>0$), and the size of the red (resp. green) contribution is proportional to the value of $f_1^v(\ell,m,n)$ (resp. $f_2^v(\ell,m,n)$).
	Summarising Section~\ref{s:ba} in visual language, we are interested in minimising the number of mixed bubbles, and maximising the number of pure bubbles.
	
	\begin{figure}[ht!]
		\caption{\sf Value distribution for the triples $(a_3,a_5,a_7)$ (left) and $(a_2,a_3,a_5)$ (right) where red (resp. green) bubbles correspond to class number 1 (resp. 2) real quadratic fields.}
		\label{fig:bubble_chart}
		\begin{adjustwidth}{-.5in}{-.5in} 
			\centering
			\subfloat{\includegraphics[width=9cm]{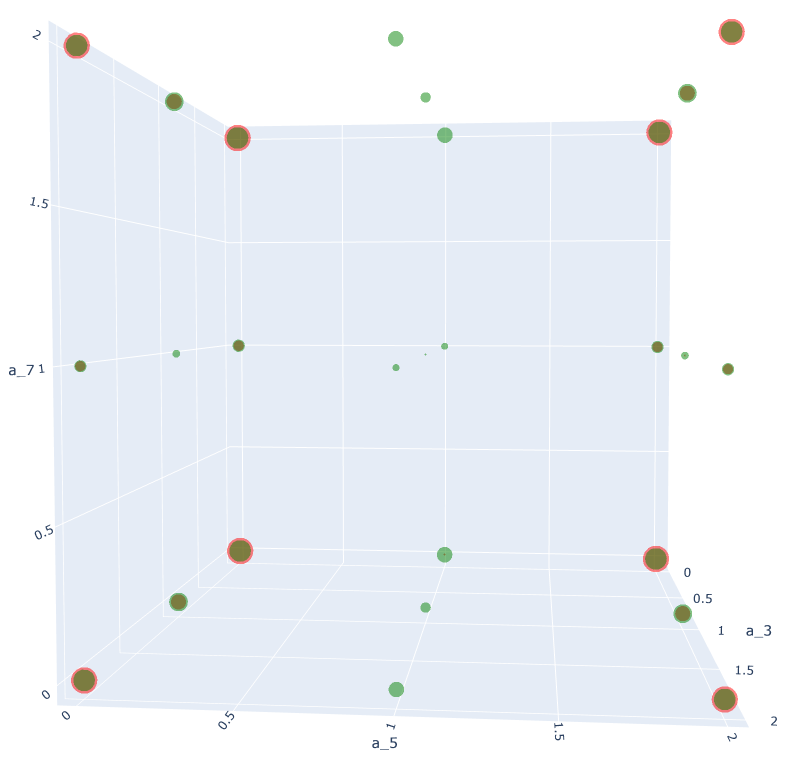}}
			\subfloat{{\includegraphics[width=9cm]{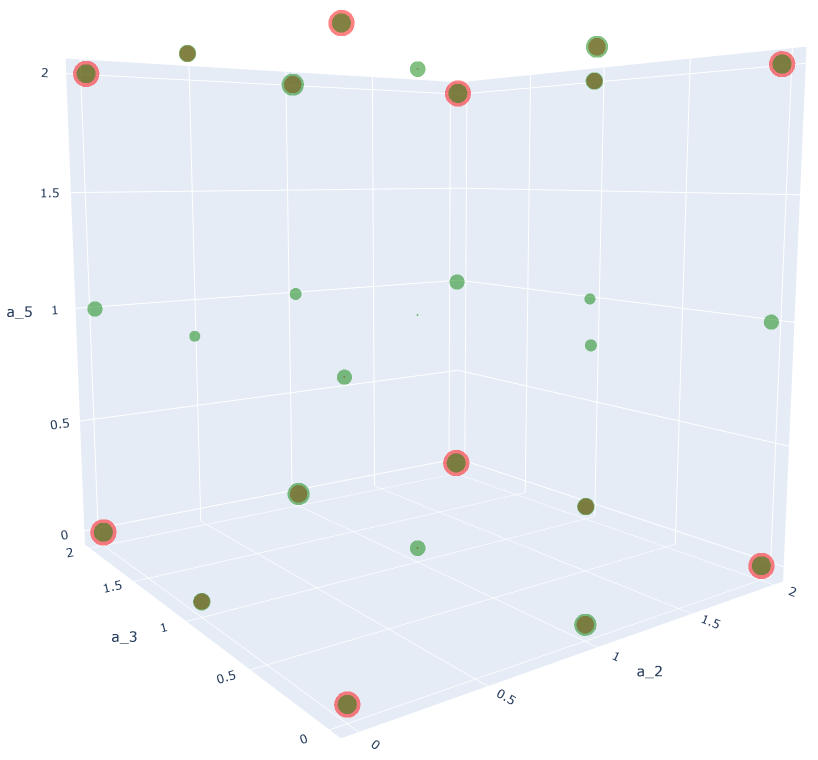} }}
		\end{adjustwidth}
	\end{figure}
	
	In order to compute the counting functions of Section~\ref{ss:cf}, and hence the cost function of Section~\ref{s:ba}, we define a large matrix whose rows are indexed by real quadratic fields $K_d$. 
	The first column contains $d$,
	the second column contains its class number $h_d$, 
	and the remaining columns contain the values taken by the coefficients $a_i$ of $\zeta_d(s)$. 
	If the matrix contains $c$ rows in which $(a_{\ell},a_m,a_n)=(x,y,z)$, then we say that the triple $[\ell,m,n]$ provides $c$ collisions for the vector $(x,y,z)$.
	
	\begin{table}[ht!]
		\begin{center}
			\caption{\sf Sample of the matrix used to minimise $C_{1,2}$.}\label{tab:collision}
			\begin{tabular}{|c|c|c|c|c|c|c|c|} 
				\hline
				$d$ & \small\text{class number} & ${a_1}$ & $a_2$ & $a_3$ & ${a_4}$ & $a_5$ & $a_7$\\\hline \hline
				${5}$ & 1 & 1 & 0 & 0 & 1 & 1 & 0\\ \hline
				${33}$ & 1 & 1 & 2 & 1 & 3 & 0 & 0\\ \hline
				${61}$ & 1 & 1 & 0 & 2 & 1 & 2 & 0\\ \hline
				${10}$ & 2 & 1 & 1 & 2 & 1 & 1 & 0\\ \hline
				${15}$ & 2 & 1 & 1 & 1 & 1 & 1 & 2 \\ \hline 
				${65}$ & 2 & 1 & 2 & 0 & 3 & 1 & 2\\ \hline
				${1309}$ & 2 & 1 & 0 & 2 & 1 & 2 & 1\\ \hline
			\end{tabular}
		\end{center}
	\end{table}

	\begin{ex}
		A very small section of this matrix is shown in Table~\ref{tab:collision}.
		Looking at Table~\ref{tab:collision}, we see that the triple $[1,2,4]$ provides three collisions for the vector $(1,0,1)$ and two collisions for the vector $(1,2,3)$. 
		On the other hand, the triple $[2,4,7]$, yields no collisions at any vector.
	\end{ex}

	For each triple $[\ell,m,n]$, we search the entire LMFDB to generate an output of the form:
	\[ (g_1(\ell,m,n),g_{1,2}(\ell,m,n), g_2(\ell,m,n), C_{1,2}(\ell,m,n)). \]

	\begin{ex}\label{ex665740985}
		Consider the triple $[\ell,m,n]=[665,740,985]$. 
		In this case, the output of the bubble algorithm reads as $(18,18,57,0.461538)$.
		In other words, $(a_{665},a_{740},a_{985})$ takes $18$ different values for class number one fields, 
		$18$ common values for both class numbers,
		$57$ different values for class number two fields, 
		and has a cost of $0.461538$. 
		In particular, there is no value for $(a_{665},a_{740},a_{985})$ taken by a class number 1 field which is not also taken by a class number 2 field.
	\end{ex}
	
	Note that there may exist many pure bubbles of one colour and very few or zero pure bubbles of the other. More precisely, the number of pure red (class number 1) bubbles is given by 
	$g_1(\ell,m,n)-g_{1,2}(\ell,m,n)$ and that of pure green (class number 2) ones by $g_2(\ell,m,n)-g_{1,2}(\ell,m,n)$.  Thus, in Example~\ref{ex665740985}, there are no pure red bubbles and $57-18=39$ pure green bubbles for $[\ell,m,n]=[665,740,985]$.
	
	Actually, in the entire dataset available at LMFDB, there exist triples $[\ell, m,n]$ yielding 109 pure green bubbles with no pure red bubbles.
	In Table~\ref{tab:corner}, we list the maximal number of pure green bubbles conditional on the constraint that there is a fixed small number of pure red bubbles.
	
	\begin{table}[ht!]
		\begin{center}
			\caption{\sf Maximum number of pure green bubbles (class number 2), conditional on a prescribed small number of pure red bubbles (class number 1).}\label{tab:corner}
			\begin{tabular}{r||c|c|c|c|c} \label{corners}
				\# of pure red bubbles & 0 & 1&2&3&4\\ \hline
				max \# of pure green bubbles &109&80&60&48&33
			\end{tabular}
		\end{center}
	\end{table} 
	
	Given 3 pure red bubbles, there is a unique triple achieving the maximum of 48 pure green bubbles. Specifically, the triple is $[691,693,850]$. In general, a triple achieving the maximum need not be unique. For example, given 1 pure red bubble, there are 6 triples achieving the maximum of 80 pure green bubbles, as listed in Table~\ref{tab:1red}.
	\begin{table}[ht!]
		\begin{center}
			\caption{\sf Triples $[\ell,m,n]$ whose value distribution yields 1 pure red bubble and $80$ pure green bubbles. In all cases, note that $g_1(\ell,m,n)-g_{1,2}(\ell,m,n)=1$ and $g_2(\ell,m,n)-g_{1,2}(\ell,m,n)=80$.}
			\label{tab:1red}
			\begin{tabular}{c|c|c||c|c|c|c} 
				$\ell$ & $m$ & $n$ & $g_1$ & $g_{1,2}$ & $g_2$ & $C_{1,2}$\\
				\hline
				589 & 637 & 720&48&47&127&0.580247 \\
				637 & 720 & 989&48&47&127&0.580247\\
				585 & 620 & 931&64&63&143&0.777778\\
				372 & 931 & 975&83&82&162&1.012346\\
				804 & 931 & 975&83&82&162&1.012346\\
				775 & 819 & 987&84&83&163&1.024691
			\end{tabular}
		\end{center}
	\end{table}

	\begin{table}[t!]
		\begin{center}
			\caption{\sf Triples $[\ell,m,n]$ minimising the cost function in the case that $\{i,j\}=\{1,3\}$. In all cases, we have $C_{1,3}(\ell,m,n)=5/3$.}
			\label{tab:best13}
			\begin{tabular}{c|c|c||c|c|c} 
				$\ell$ & $m$ & $n$ & $g_1(\ell,m,n)$ & $g_{1,3}(\ell,m,n)$ & $g_3(\ell,m,n)$ \\
				\hline
				62&904&1120&14&10&12\\
				65&1166&1868&18&15&21\\
				65&1300&1604&19&15&20\\
				65&1316&1820&18&15&21\\
				258&1456&1784&19&15&20\\
				258&1580&1784&17&15&22\\
				262&1324&1616&12&10&14\\
				262&1280&1844&14&10&12\\
				269&1436&1844&21&15&18\\
				274&1316&1820&12&10&14\\
				274&1324&1576&12&10&14\\
				289&1436&1576&19&15&20\\
				1364&1568&1913&26&20&26\\
				1374&1444&1664&20&15&19\\
				1468&1802&1984&20&15&19\\
			\end{tabular}
		\end{center}
	\end{table}

	\subsection{Challenges in the case $h_d\in\{1,3\}$}\label{s:ba13}
	The bubble algorithm naturally generalises to other binary classification problems for class numbers in which the features are given by $\zeta$-coefficients. When applied to the dataset of $h_d \in \{1,3\}$, the best performing triples $[\ell,m,n]$ computed by the bubble algorithm are summarised in Table~\ref{tab:best13}.
	Note that the minimal value of achieved by the cost function (which is $5/3$) is much larger than the value previously seen for $h_d\in\{1,2\}$ in Example~\ref{ex665740985}. This quantifies one way in which the classification of $h_d\in\{1,3\}$ is a fundamentally more challenging problem using only $\zeta$-coefficients. 
	Furthermore, when $h_d\in\{1,2\}$, we saw that the value distribution for the triple $[3,5,7]$, which consists of prime numbers, yields some pure green bubbles (cf. Figure~\ref{fig:bubble_chart}).
	On the other hand, in the case that $h_d\in\{1,3\}$, the bubble algorithm does not yield a single such triple of prime numbers. Instead, the optimal values of the function $C_{1,3}$ listed in Table~\ref{tab:best13} are taken by triples of composite numbers. 
	In Section~\ref{s:13}, we will circumvent these challenges by introducing training sets with additional features.
	
	\section{Class numbers 1 and 2}\label{s:12}
	In this section, we investigate the binary classification of real quadratic fields with class numbers 1 and 2 using gradient-boosting tree-based learning algorithms and genetic programming. 
	In Section~\ref{ss:pi}, we apply the LightGBM and CatBoost machine learning algorithms to finite lists of $\zeta$-coefficients. 
	We report all metric scores in the form of tables and figures, including the list of most important features used for the predictions in each model.
	In Section~\ref{sym_class_K}, we use a genetic programming algorithm called symbolic classification to the ramified primes to obtain an optimal approximation for the class number formula, and subsequently recover some results about the parity of the class number first presented in Section ~\ref{s:theory}.
	We maintain the notation from Section \ref{s:theory}.
	
	\subsection{Learning $h_d\in\{1,2\}$ from the prime index coefficients of $\zeta_d(s)$}\label{ss:pi}
	To each square-free $d\in\mathbb{Z}_{>0}$, we attach the vector
	\begin{equation}\label{eq.vectors}
	v(d)=(a_p)_{\substack{p\text{ prime}\\p\leq 1000}}\in\mathbb{Z}^{168},
	\end{equation}
	where $a_p$ is the $p$th Dirichlet coefficient of $\zeta_d(s)$.
	The dimension 168 is the number of primes $\leq1000$. 
	Using equation~\eqref{eq.ap}, we observe that $v(d)\in\{0,1,2\}^{168}$.
	Using the vectors in equation~\eqref{eq.vectors}, we introduce the labelled dataset
	\begin{equation}\label{eq.LDS}
	\mathcal{D}_{1,2}=\{v(d)\rightarrow h_d\},
	\end{equation}
	where $d$ varies over square-free positive integers such that $h_d\in\{1,2\}$ 
	and $d$ appears in the LMFDB \cite{lmfdb}. 
	As was noted in Remark~\ref{rem.LMFDB}, the dataset is complete for $d$ such that $D<2\times10^6$, and contains some larger $d$ (the largest $d$ such that $h_d\in\{1,2\}$ being $43,723,857$).
	
	The labelled dataset in \eqref{eq.LDS} is different to that in \cite[Section~6.1]{HLOb}, which also incorporated the $\zeta$-coefficients with composite indices.
	Our choice to include only prime indices seems intuitively reasonable, since all coefficients can be recovered from those of prime index via \eqref{eq.DedZet}. Whilst it was observed in Section~\ref{s:ci} (in particular, Example~\ref{e:=0}) that $\zeta$-coefficients with composite indices yield greater separation between real quadratic fields of different class numbers, prime indices are nevertheless sufficient for the high accuracy classifiers in the context of this section. We furthermore note that the exclusion of composite indices can also be motivated by the correlation matrix in Figure~\ref{fig:my_label}, which shows that composite indices sharing similar prime decomposition are highly correlated.
	Excluding the composite indices therefore allows for faster training, better generalization, and more reliable (permutation) feature importance. 
	
	\begin{figure}[b!]
		\centering
		\caption{\sf Correlation matrix of the first ten coefficients of the Dedekind zeta functions $\zeta_d(s)$ of real quadratic fields with $h_d\in\{1,2\}$.}
		\label{fig:my_label}
		\includegraphics[width=10cm]{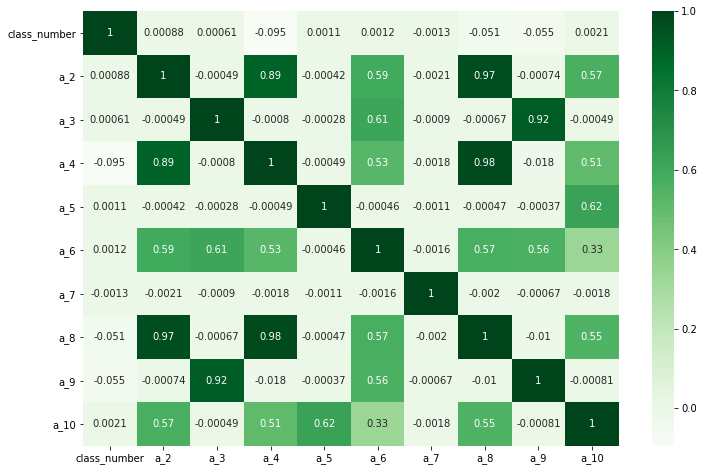}
	\end{figure}
	
	To build our LightGBM and CatBoost supervised learning models, we will use the automated machine learning library AutoMLjar \cite{MLjar}. Our choice to use LightGBM and CatBoost was motivated in large part by Figure~\ref{fig:ldb_performance}, which shows that LightGBM and CatBoost are, in addition to being fast to train, among the best performing models that can be built using AutoMLjar for our experiment. Furthermore, they often represent state-of-the-art models on tabular data. We will consider a training/testing split of 70/30 for $\mathcal{D}_{1,2}$, together with a $10$-fold cross validation performed on the training set, since other splits including 30/70 produce similar results. All the folders generated by AutoMLjar, including codes, datasets and figures for all experiments presented in this article can be found in the GitHub repository \cite{Malik_GH}.
	Tables \ref{tab:70-30-train-metrics}  summarizes various performance scores for the LightGBM and CatBoost models, and  Figure~\ref{fig:ks-LightGBM} summarizes the KS statistics.
	
	\begin{figure}[ht!]
		\caption{\sf Comparison of the LightGBM and CatBoost learning models against other AutoMLjar models on a 70/30 split for the binary classification task $h_d=1$ vs $h_d=2$.}
		\label{fig:ldb_performance}
		\begin{adjustwidth}{-.5in}{-.5in} 
			\centering
			\subfloat{{\includegraphics[width=9cm]{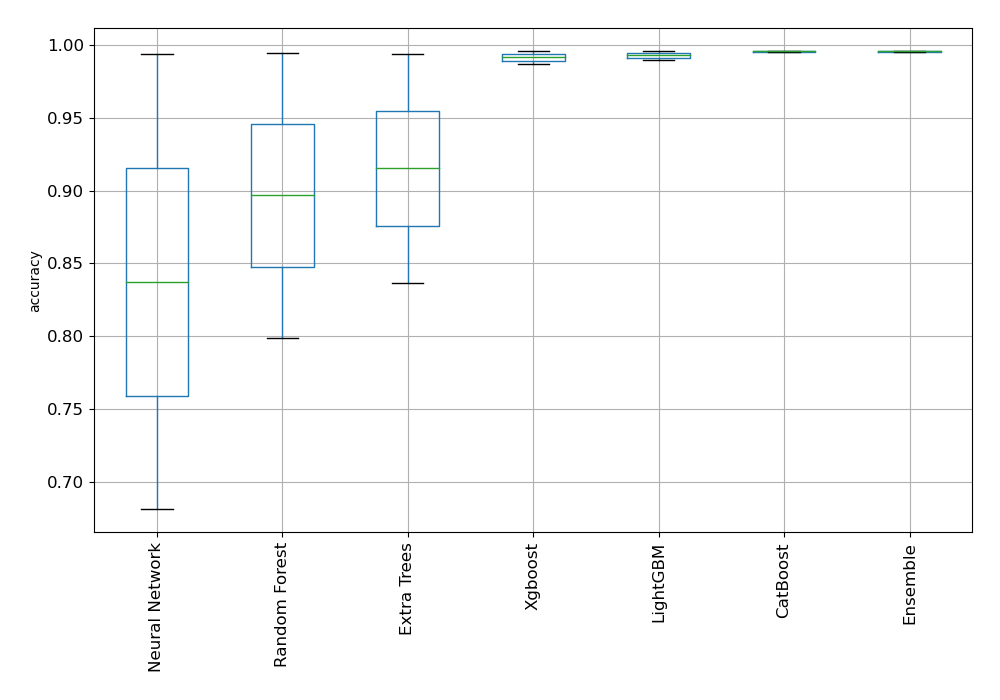} }}
			
		\end{adjustwidth}
	\end{figure}
	
	\begin{table}[ht!]
		\begin{center}
			\caption{\sf Performance metrics over the training set of the LightGBM and CatBoost model for the binary classification task $h_d=1$ vs $h_d=2$ with split $70 / 30$.
				As is standard, AUC is the "area under the receiver operating characteristic curve", F1 is the F1-score and MCC, the Matthews correlation coefficient.
				All these quantities need to be close to 1 for a good model. LogLoss is the log of the loss function and needs to be close to 0.
			}
			\label{tab:70-30-train-metrics}
			\begin{tabular}{c|c|c|c|c|c} 
				\textbf{Model}& \textbf{Logloss}&\textbf{AUC}& \textbf{F1}&\textbf{Accuracy}&\textbf{MCC}\\
				\hline
				\textit{LightGBM} &0.0687&0.99776& 0.9902&0.9902&0.9804 \\\textit{CatBoost} &0.0323&0.9992&0.9952&0.9951&0.9903
			\end{tabular}
		\end{center}
	\end{table} 
	
	\begin{figure}[ht!]
		\caption{\sf KS statistic plot over the training set of the LightGBM (right) and CatBoost model (left) for the binary classification task $h_d=1$ vs $h_d=2$ with split $70 / 30$.}
		\label{fig:ks-LightGBM}
		\begin{adjustwidth}{-.5in}{-.5in} 
			\centering
			\subfloat{{\includegraphics[width=9cm]{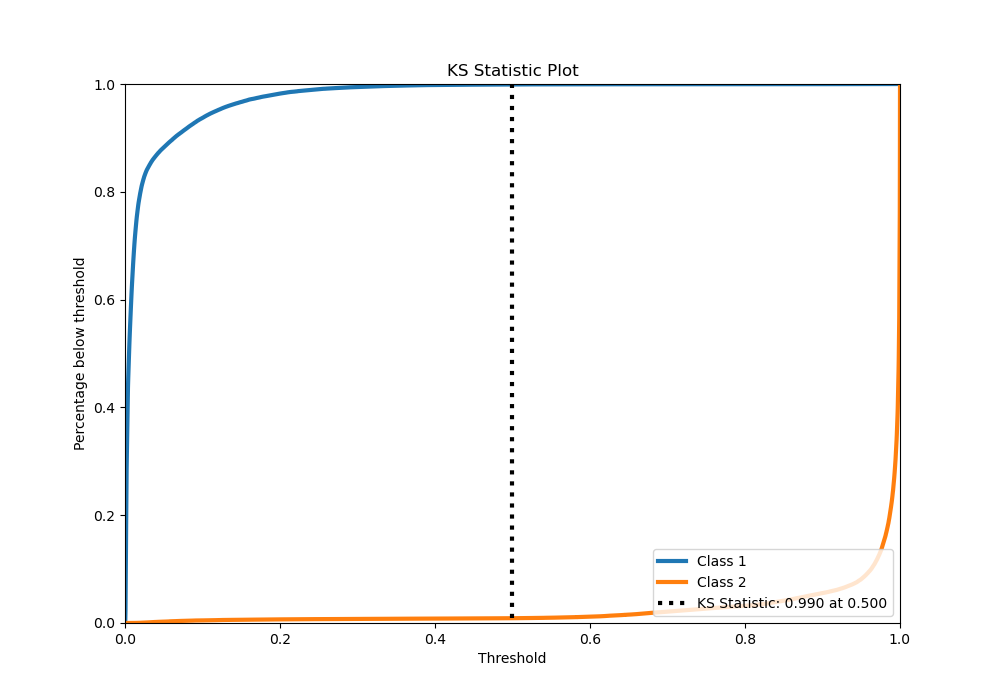} }}
			\subfloat{{\includegraphics[width=9cm]{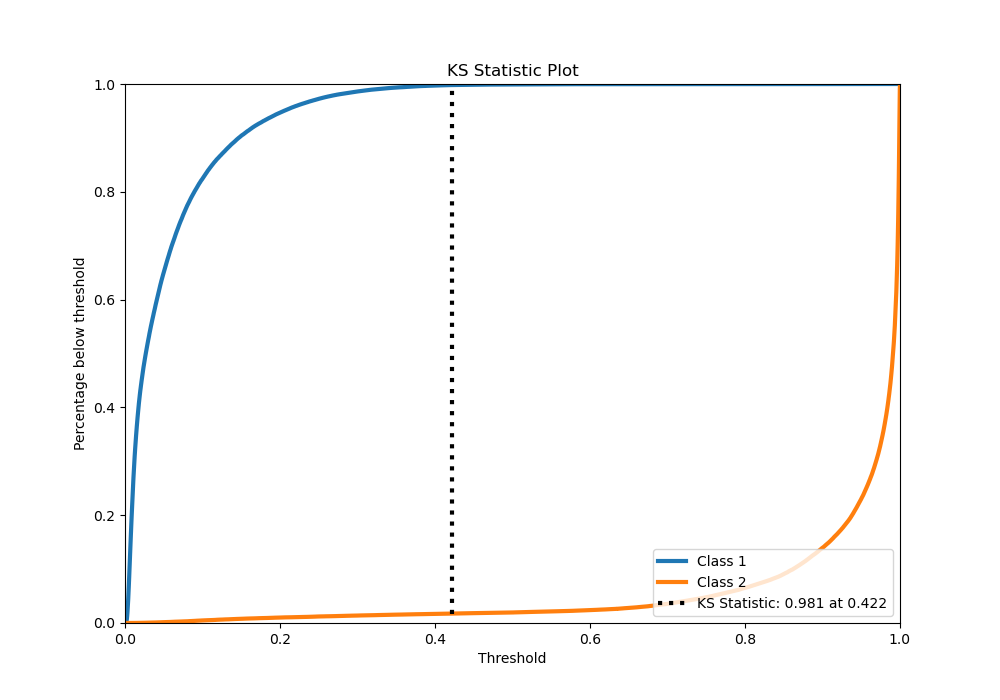} }}
		\end{adjustwidth}
	\end{figure}
	
	Figure~\ref{fig:feature_importance_70-30} summarizes the permutation feature importance.
	Looking at Figure~\ref{fig:feature_importance_70-30}, we observe that the small prime index coefficients are the most important features for the prediction of the class number. 
	In particular, the triple $(a_2,a_3,a_5)$ is the most important for all models. 
	Table \ref{tab:a2a3a5} gives the value distribution of this triple, from which we can clearly see an unequal distribution depending on whether $h_d=1$ or $h_d=2$. 
	This unequal distribution of the data is much more pronounced in the case where $a_2=a_3=a_5=1$ and it is visualized in Figure~\ref{fig:bubble_chart}. 
	
	This points towards the distribution of ones in the sequences of prime coefficients, namely, towards the number $n_d$ of primes that ramify in $K_d$ and that can be captured in the first 1000 coefficients of the function $\zeta_d(s)$. This is in line with the lemmas in Section \ref{s:theory}. Table \ref{tab:ramified_primes} shows the distribution of ramified primes against the class number and Table \ref{tab:prime_detected} represents the distribution of ramified primes appearing in the first 1000 coefficients of the function $\zeta_d(s)$.
	
	\begin{figure}[ht!]
		\caption{\sf Permutation feature importance for the LightGBM (left) and Catboost (right) model for the binary classification task $h_d=1$ vs $h_d=2$ on a $70/30$ split.}
		\label{fig:feature_importance_70-30}
		\begin{adjustwidth}{-.5in}{-.5in} 
			\centering
			\subfloat[\centering LightGBM]{\includegraphics[width=9cm]{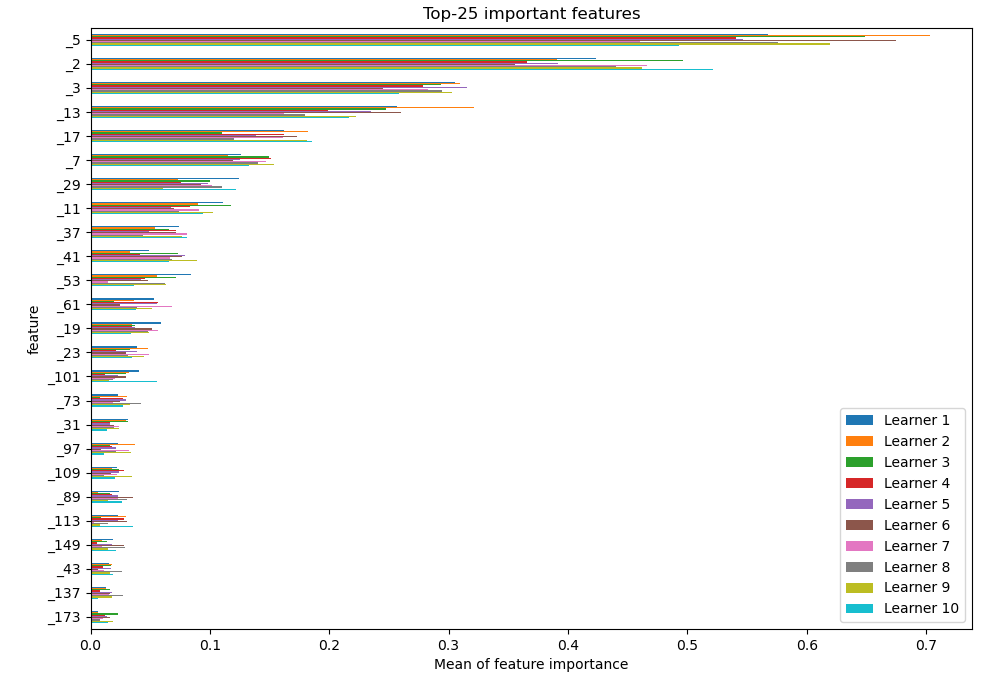}}
			\subfloat[\centering CatBoost]{{\includegraphics[width=9cm]{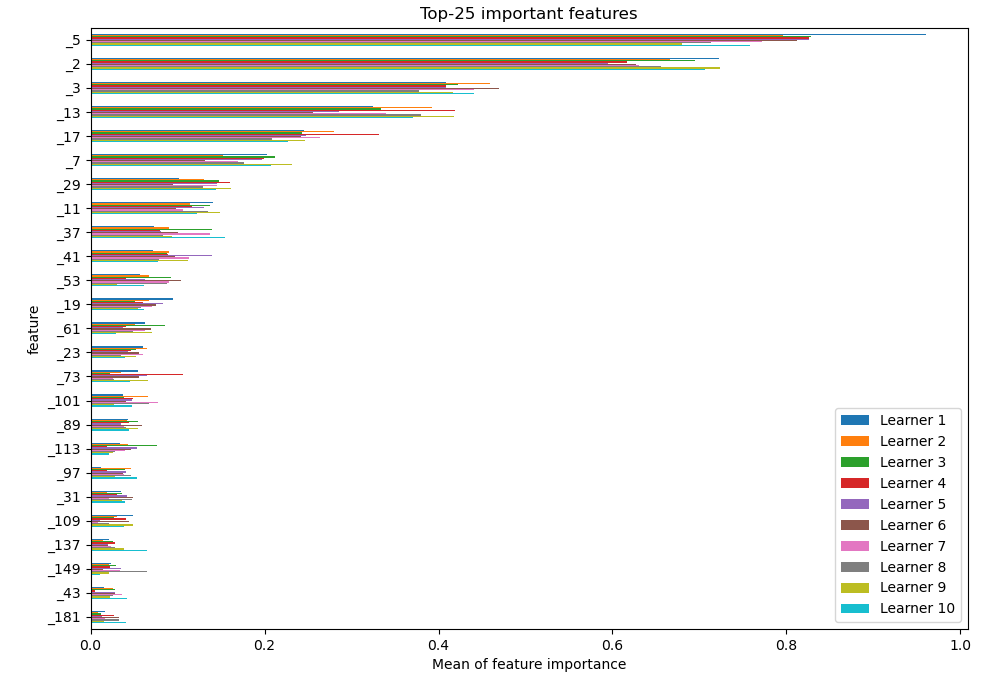} }}
		\end{adjustwidth}
	\end{figure}

	\begin{table}[ht!]
		\begin{center}
			\caption{\sf Number of real quadratic fields with class number $h_d \in \{1,2\}$ in dataset $\D_{1,2}$ for specified values of the $\zeta$-coefficients $a_2,a_3,a_5$}
			\label{tab:a2a3a5}
			\begin{tabular}{c||c|c|c||c|c|c||c|c|c} 
				\textbf{$h_d$} & \textbf{$a_2=0$} & \textbf{$a_3=0$} & \textbf{$a_5=0$}& \textbf{$a_2=1$} & \textbf{$a_3=1$} & \textbf{$a_5=1$}& \textbf{$a_2=2$} & \textbf{$a_3=2$} & \textbf{$a_5=2$}\\
				\hline
				1 & 74868 & 76965&88786& 27804 & 23650& 1& 74487& 76544&88372\\
				2 & 60392 & 69501&75225& 62773 & 44673&33013& 60271 & 69262&75198\\
			\end{tabular}
		\end{center}
	\end{table}

	\begin{table}[ht!]
		\begin{center}
			\caption{\sf Number of real quadratic fields with class number $h_d \in \{1,2\}$ in dataset $\D_{1,2}$ for specified number $n_d$ of ramified primes}\label{tab:ramified_primes}
			\begin{tabular}{c|c|c|c} 
				\textbf{$h_d$} & $n_d=1$ & $n_d=2$ & $n_d=3$\\
				\hline
				1 & 64522 & 112637&0\\
				2 & 0 & 52451&130985\\
			\end{tabular}
		\end{center}
	\end{table} 
	
	\begin{table}[ht!]
		\begin{center}
			\caption{\sf Detected ramified primes in the first 1000 coefficients of the function $\zeta_d(s)$ for fields $K_d$ with class number $h_d=1,2$ from dataset $\D_{1,2}$.}
			\label{tab:prime_detected}
			\begin{tabular}{c|c|c} 
				\text{$h_d$}&\textbf{\# of detected ramified primes}& \textbf{\# of fields with this property}\\
				\hline
				&0&65468\\
				1&1&108527\\
				&2&3164\\
				\hline
				&0&806\\
				2&1&49595\\
				&2&101978\\
				&3&31057\\
			\end{tabular}
		\end{center}
	\end{table}

	\subsection{Learning $h_d\in\{1,2\}$ with symbolic classification}\label{sym_class_K}
	
	In this section, we will construct a learning model obtained from an genetic programming algorithm called symbolic classification \cite{GPSC}. 
	We will explore how symbolic classification can be used to produce an explicit predictor for the class number of real quadratic fields in our dataset, using the software HeuristicLab \cite{HLab}.
	Such a predictor will provide an approximation to the class number formula for our dataset. This approach has been used in theoretical physics, e.g. \cite{AIfeynmann}, to develop good approximations for certain physical quantities based on a given set of learning features. In our setting, we will see that reasonable approximations to the class number formula \eqref{eq.ACNF} can be discovered using simpler learning features. Such approximations, as we will see, are able to shed light on interesting properties of real quadratic fields. In particular, they will allow us to recover some results in Section \ref{s:theory} on the parity of the class number.  
	
	Using the real quadratic fields of Section \ref{ss:pi}, we create a labeled dataset $\D^{\mathrm{SC}}_{1,2}$ by considering as learning features the number $n_d$ of ramified primes in $K_d$ and the ramified primes themselves. 
	Since $h_d \in \{1,2 \}$, we have $n_d\leq3$ by Lemma \ref{l:3rp} and write the ramified primes as $p_1,p_2,p_3$. 
	By convention, if there is no second (resp. third) ramified prime, we will set $p_2=0$ (resp. $p_3=0$). 
	Our dataset $\D^{\mathrm{SC}}_{1,2}$ can be summarized as follows 
	\[\mathcal{D}^{\mathrm{SC}}_{1,2}=\{(n_d,p_1,p_2,p_3)\rightarrow h_d\}.\] 
	As the dataset $\D^{\mathrm{SC}}_{1,2}$ is quite large, running symbolic classification becomes computationally demanding. 
	We thus sample randomly $50,000$ data points from $\mathcal{D}^{\mathrm{SC}}_{1,2}$ and construct a dataset where fields of class number 1 and 2 are chosen evenly. 
	In this section, we use a training and testing split of $40/60$. 
	We will apply our approximate formulas to the whole LMFDB and test their performances. 
	
	Our symbolic classifier was built using the following parameters of HeuristicLab. The population size was fixed to 100, the fitness function used was \textit{mean squared error}, the crossover method used was \textit{subtree swapping crossover}, the mutator was \textit{multi symbolic expression tree manipulator} and for \textit{elitism}, we kept one elite at every generation to favor population exploration and avoid exploitation. 
	For the alphabet of functions, we used various mix of the available functions.

	Among all the formulas we have generated, the simplest one in terms of length and depth of the model is the following approximation of the class number formula: 
	\begin{equation}\label{f3}
	h^{(1,2)}_d=0.17257 \,\sin(1.5703\, p_1)+0.72004 \, n_d, 
	\end{equation}
	Equation~\eqref{f3} yields the following predictor $\phi(d)$ for $h_d$, which has accuracy $98\%$ on the training and testing set, and whose accuracy persists when applied to the entire dataset $\D^{\mathrm{SC}}_{1,2}$:
	\begin{equation}\label{eq.pred}
	\phi(d)=\begin{cases}1& \text{ if } h^{(1,2)}_d<t,\\2& \text{ if } h^{(1,2)}_d\geq t,\end{cases} \ \quad \text{ for } t=1.5115.
	\end{equation}
	The value $t$ appearing in \eqref{eq.pred} serves as the threshold for \eqref{f3}. 
	The classification metrics for \eqref{eq.pred} are listed in Figure \ref{fig:fi_metricsf3}. Class number 1 fields denote the positive class and Figure \ref{fig:fi_metricsf3} shows in particular that our formula predicts no false negatives, very few false positives, and is thus well-suited to study the relationship between $p_1,n_d$ and $h_d$ from the data.
	\begin{figure}[ht!]
		\caption{\sf Summary of the classification metrics for \eqref{f3} used in the binary classification task $h_d=1$ vs $h_d=2$.}
		\label{fig:fi_metricsf3}
		\vskip 0.1 cm 
		\begin{adjustwidth}{-.5in}{-.5in} 
			\centering
			\subfloat[]{{\includegraphics[width=7cm]{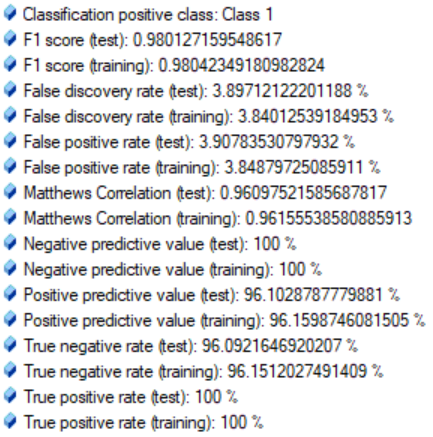} }}
		\end{adjustwidth}
	\end{figure}
	
	Looking at equation~\eqref{f3}, we note that only the number $n_d$ of ramified primes and the first ramified prime $p_1$ are selected by the genetic algorithm and sufficient to distinguish class number 1 and 2 real quadratic fields with very high accuracy. 
	It is natural to ask what kind of information we can extract from $h^{(1,2)}_d$?
	
	If $n_d=1$ (resp. $n_d=3$) then it is easy to see that $\phi(d)=1$ (resp. $\phi(d)=2$) from \eqref{f3}, regardless of the value of $p_1$. Thus we recover Lemma~\ref{l:1rp} and Lemma~\ref{l:3rp} (both under the additional hypothesis that $h_d\in\{1,2\}$), and $\phi$ distinguishes between class number 1 and 2 fields with $100\%$ accuracy. In the case $n_d=2$, we have the following conditions imposed on $p_1$:
	\begin{equation}\label{eq.F3t3}
	h^{(1,2)}_d<t\iff 4n-2.2724<p_1<4n+0.27174, \ \ n\in\mathbb{Z}_{>0}.
	\end{equation}
	By analyzing equation~\eqref{eq.F3t3}, we find that the allowed primes are those $p_1$ congruent to $3\bmod 4$ and $p_1=2$. If, instead, we investigate the $h^{(1,2)}_d\geq t$ condition, then we find the missing primes $p_1$ congruent to $1\bmod 4$. 
	In other words, we have recovered Lemma~\ref{l:srp}.
	All together it reflects the general fact about parities given in Corollary \ref{iff}.
	
	Though a data point in $\mathcal{D}^{\mathrm{SC}}_{1,2}$ contains $p_1, p_2, p_3$, the formula in \eqref{f3} has only the first ramified prime $p_1$, which agrees with Lemma~\ref{l:srp} (in the case $n_d=2$). It is interesting to see that our machine learning model has \textit{learned} this fact from data alone.

	We will see in Section \ref{h=1vs3} that such a classification failed in the case of class number 1 and 3 fields. This gives evidence that one cannot expect to distinguish fields with the same parity using only the simple features $n_d$ and the list of ramified primes. 
	\section{Class numbers 1 and 3}\label{s:13}
	
	We now turn to the classification of real quadratic fields of class number 1 and 3. 
	We attempt to mimic the structure of Section~\ref{s:12}, but include new ideas to circumvent the difficulties we encounter.
	Based on the cost function analysis undertaken in Section~\ref{ss.sep}, we may suspect that $\zeta$-coefficients alone do not yield high accuracy classifiers in this case.
	This is borne out in our implementation, in which, trained on these features, neither LightGBM nor Symbolic Classification succeeded in distinguishing real quadratic fields of class number 1 or 3. 
	Subsequently, we incorporate other features, such as the regulator and partial sums.
	
	\subsection{Balancing data and selecting features}\label{h=1vs3}
	There is a large imbalance between the 25,201 real quadratic fields with class number 3 and the 177,159 with class number 1 included in \cite{lmfdb}. 
	To avoid any bias in the construction of our dataset, we will restrict ourselves to the 11,531 real quadratic fields of class number 3 with discriminant $D\leq 10^6$. 
	To construct a balanced dataset including the same number of real quadratic fields with class number 1 evenly distributed over discriminant ranges, we sample class number 1 fields randomly from the sets $\{K_d:i\cdot 10^5\leq D\leq (i+1)\cdot10^5\}$, $i\in\{0,\dots,9\}$, as many as class number 3 fields in each interval. 
	For example, there are 1,261 fields with class number 3 in the interval $1\leq D\leq10^5$ and so in this interval, we pick the same amount of class number 1 fields. 
	
	Our dataset thus consists of 23,062 fields. 
	For each of them, we compute the vector $v(d)$ as in \eqref{eq.vectors}, 
	the discriminant $D$, 
	the regulator $R_d$, 
	the number $n_d$ of ramified primes,
	the ramified primes $p_1,p_2$ where we set $p_2=0$ in the case that only one prime ramifies,
	and the partial sums $S_{\zeta_d}$ and $S_{\chi_D}$ defined by
	\[ S_{\zeta_d} := \sum_{n=1}^{1000} \frac {a_n} n \quad \text{ and } \quad S_{\chi_D}:= \sum_{n=1}^{1000} \frac {\chi_D(n)} n . \] The partial sums $S_{\zeta_d}$ and $S_{\chi_D}$ are motivated by $\zeta_d(s)$ and $L(s, \chi_D)$, respectively. Our dataset can thus be summarized as
	\[\mathcal{D}^{\mathrm{SC}}_{1,3}=\{(v(d),D,R_d,S_{\zeta_d}, S_{\chi_D},n_d,p_1,p_2)\rightarrow h_d\}.\]

	\subsection{Learning $h_d\in\{1,3\}$ from $\mathcal{D}^{\mathrm{SC}}_{1,3}$}\label{13sl}
	
	The problem of classifying real quadratic fields of class number 1 and 3 is much more difficult than the problem of classifying those with class number 1 and 2. Section~\ref{s:ba13} discussed this problem from a conceptual perspective. In this section, we observe the increased difficulty from a practical standpoint.
	
	In Table \ref{tab:class_number_1vs3} we summarise the features and testing accuracy for various experiments using the LightGBM classifier. 
	As in Section \ref{s:12}, the data segregation is $70/30$, together with a $10$-fold cross validation performed on the training set.
	
	\begin{table}[ht!]
		\begin{center}
			\caption{\sf Performance of LightGBM on various combinations of learning features for the binary classification task of $h_d=1$ vs $h_d=3$.}
			\label{tab:class_number_1vs3}
			\begin{tabular}{|c|c|c|c} 
				\textbf{Row Number}&\textbf{Features}&\textbf{Testing Accuracy}\\
				\hline
				1&$(a_p)_{p\leq 1000}$ &  53.34\% \\
				\hline
				2&$(a_p)_{p\leq 1000},~n_d,~p_i$ & 54.87\% \\ \hline
				3&$(a_p)_{p\leq 1000},~S_{\zeta_d}$ & 52.87\%  \\ \hline
				4&$(a_p)_{p\leq 1000},~R_d$ & 91.21\% \\ \hline
				5&$(a_p)_{p\leq 1000},~D$ & 53.75\%\\ \hline
				6&$(a_p)_{p\leq 1000},~D,~R_d$ & 99.38\% \\ \hline
				7&$D,~R_d,~S_{\zeta_d}$ & 99.84\% \\ \hline
				8& $(a_p)_{p\leq 1000},~S_{\chi_D}$ &53.45\% \\ \hline
				9& $D,~R_d,~S_{\chi_D}$& 99.93\% \\ \hline
				10&$(a_{p_i})_{i=1}^{10},~D,~R_d$ & 99.90\%\\ \hline
				11&$(a_{p_i})_{i=1}^{5},~D,~R_d$ & 99.54\% \\ \hline
				12&$(a_{p_i})_{i=1}^{3},~D,~R_d$ & 99.55\% \\ \hline
				13&$(a_{p_i})_{i=1}^{2},~D,~R_d$ &97.51\%\\ \hline
				14&$a_{p_1},~D,~R_d$ & 88.46\%\\ 
			\end{tabular}
		\end{center}
	\end{table} 
	
	There are various observations we can make from Table \ref{tab:class_number_1vs3}. Row 4 shows that the regulator and the prime index coefficients are strong predictors for the class number. 
	When we add the discriminant in row 6, we obtain an almost perfect classification. Also look at the following rows where high accuracies are attained with $S_{\zeta_d}$ and $S_{\chi_D}$.
	Of course, this could be expected from the class number formula \eqref{eq.ACNF}, and we will generate some approximate formulas to $h_d$ in Section~\ref{13sc}.
	As in Section~\ref{s:12}, rows 10-14 show the importance of the small prime index coefficients.
	
	Row 2 suggests that the values $n_d$ and the number of ramified primes detected in the sequence of prime index $\zeta$-coefficients may be uniformly distributed over our dataset of class number 1 and 3. Indeed, Table \ref{tab:prime_detected_1vs3} and \ref{tab:ramified_primes_1vs3} show this phenomenon. Consequently, $n_d$ cannot be taken as a meaningful feature for classification. This situation is in stark contrast with the case of class number 1 and 2. 
	
	\begin{table}[ht!]
		\begin{center}
			\caption{\sf Detected ramified primes in the first 1000 coefficients of the function $\zeta_d(s)$ for real quadratic fields of class number 1 and 3 in dataset $\mathcal{D}^{\mathrm{SC}}_{1,3}$.}
			\label{tab:prime_detected_1vs3}
			\begin{tabular}{c|c|c} 
				\text{$h_d$}&\textbf{\# of detected primes}& \textbf{\# of fields with this property}\\
				\hline
				&0&4250\\
				1&1&6904\\
				&2&386\\
				\hline
				&0&4284\\
				3&1&6858\\
				&2&389\\
			\end{tabular}
		\end{center}
	\end{table} 
	
	\begin{table}[ht!]
		\begin{center}
			\caption{\sf Number of ramified primes in fields $K_d$ of class number $h_d\in\{1,3\}$ from dataset $\mathcal{D}^{\mathrm{SC}}_{1,3}$ in terms of their class number.}
			\label{tab:ramified_primes_1vs3}
			\begin{tabular}{c|c|c} 
				\textbf{class number $h_d$} & $n_d=1$ & $n_d=2$\\
				\hline
				1 & 4258 & 7282\\
				3 & 4288 & 7243\\
			\end{tabular}
		\end{center}
	\end{table} 
	
	\subsection{Learning $h_d\in\{1,3\}$ with symbolic classification}\label{13sc}
	
	Using the same parameters as in Section \ref{sym_class_K}, symbolic classification yields the formulas below in this subsection. The classification metrics are listed in Figure \ref{fig:fi_metrics_1-3}.
	Unlike in Section \ref{sym_class_K}, the formulas are not tested on the entire LMFDB but only on our dataset.
	
	The first formula obtained is closely related to the class number formula \eqref{eq.ACNF}:
	\begin{equation}\label{f1_1-3}
	h^{(1,3)}_d=\frac{1}{2}\frac{\sqrt{D} \, S_{\chi_D} }{R_d}.
	\end{equation}
	The following predictor attains $100\%$ accuracy on both the training and testing data:
	\begin{equation}\label{eq.pred-13-1}
	\phi(d)=\begin{cases}1&\text{ if } h^{(1,3)}_d<t,\\3& \text{ if } h^{(1,3)}_d\geq t,\end{cases} \quad \text{ for }  t=1.963.
	\end{equation}
	This result shows the capability of symbolic classification to discover an effective formula from a dataset.

	The second formula concerns the features $a_2,a_3,a_5,D$ and $R_d$:
	\begin{equation}\label{f2_1-3}
	\tilde{h}^{(1,3)}_d=\frac{1.8858\sqrt{D}}{R_d\exp(-0.5468 a_2-0.2718a_3)\cos(\sin(-0.2556a_3))\cos(-0.1962 a_5)\cos^4(-0.1952 a_5)}.
	\end{equation}
	The following predictor attains around $99.8\%$ accuracy both on the training and test set:
	\begin{equation}\label{eq.pred-13-2}
	\phi(d)=\begin{cases}1& \text{ if } \tilde{h}^{(1,3)}_d<t,\\3& \text{ if } \tilde{h}^{(1,3)}_d\geq t,\end{cases} \quad \text{ for }  t=15.97.
	\end{equation}
	
	\begin{figure}[ht!]
		\caption{\sf Classification metrics of the formulas (\ref{f1_1-3}) and (\ref{f2_1-3})}
		\label{fig:fi_metrics_1-3} \vskip 0.1 cm
		\begin{adjustwidth}{-.5in}{-.5in} 
			\centering
			\subfloat[equation (\ref{f1_1-3}) \phantom{LLLLLLLLLLLL}]{\includegraphics[width=7.3cm]{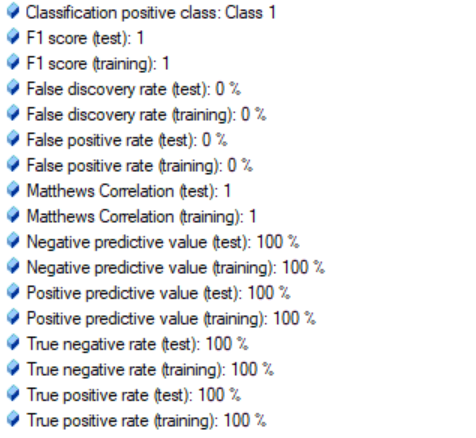}}
			\subfloat[equation (\ref{f2_1-3}) \phantom{LLLLL}]{{\includegraphics[width=7cm]{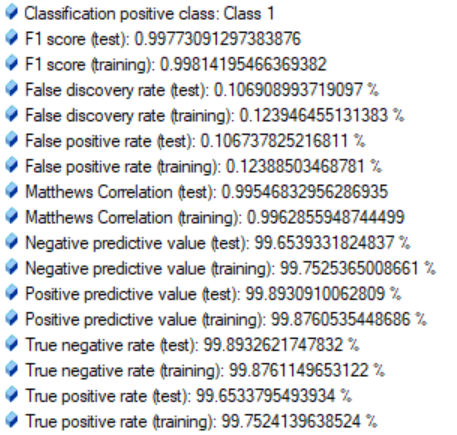}}}
		\end{adjustwidth}
	\end{figure}
	
	\newpage 
	\appendix
	\section{}
	
	\subsection{Dimensionality reduction}\label{app.dimred}
	
	In Section~\ref{ss.sep}, we explored how class number data may be separated using the \textit{bubble} algorithm.
	In this appendix, we visualize how PCA clusters our data with respect to the axes of maximal variance. To perform this dimensionality reduction, we compress the features $(a_n)_{n=1}^{1000}$ to a $3$-dimensional space.
	
	The case of class number 1 and 2 real quadratic fields is depicted in Figure \ref{fig:3d-PCA-all-data-h=1vs2}, where we observe that there are forbidden zones for class number 1 fields. It would be interesting to study explicitly these constraints. 
	In addition, we can see that the two classes are not entirely separated or mixed in this $3$-dimensional representation of the data. This shows that our data responds moderately to linear methods and points towards the fact that non linear algorithms, such as LightGBM and CatBoost, can perform better. 
	Furthermore, looking at Figure \ref{fig:3d_PCA(3)-ap-1}, we see that the prime index coefficients alone are not separable through PCA. Note that this does not imply that non-linear separability is unfeasible as we have successfully distinguished these two classes in Section \ref{s:12} using non-linear algorithms.
	
	\begin{figure}[ht!]
		\caption{\sf Dimensionality reduction using PCA for $(a_n)_{n=1}^{1000}$ in the case of real quadratic fields of class number $h_d \in \{1,2\}$ (1: blue;  2: yellow) }
		\label{fig:3d-PCA-all-data-h=1vs2}
		\begin{adjustwidth}{-.5in}{-.5in} 
			\centering
			\subfloat{\includegraphics[width=9cm]{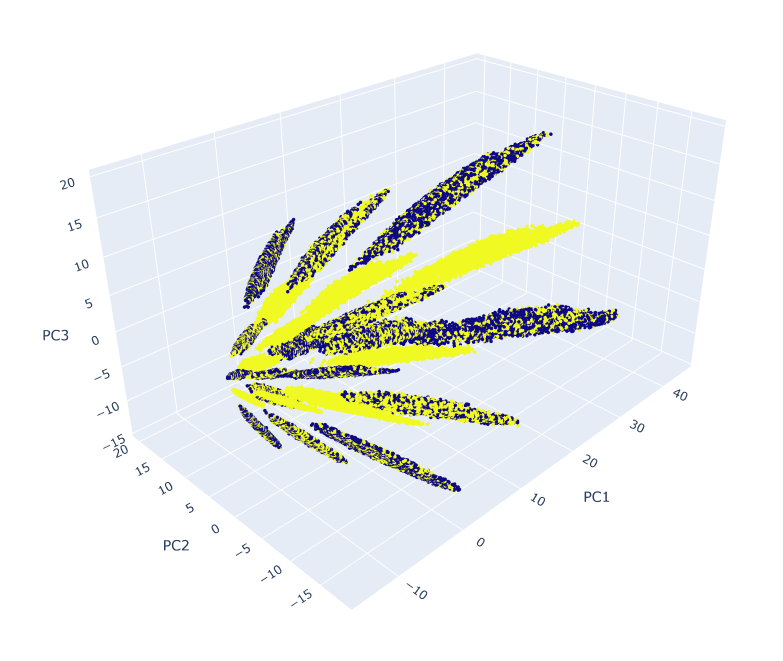}}
		\end{adjustwidth}
	\end{figure}
	
	\begin{figure}[ht!]
		\caption{\sf Dimensionality reduction using PCA for $(a_p)_{p\leq 1000 \atop p: \text{prime} } $ in the case of real quadratic fields of class number $h_d \in \{1,2\}$  (1: blue;  2: yellow)  }
		\label{fig:3d_PCA(3)-ap-1}
		\begin{adjustwidth}{-.5in}{-.5in} 
			\centering
			\includegraphics[width=8cm]{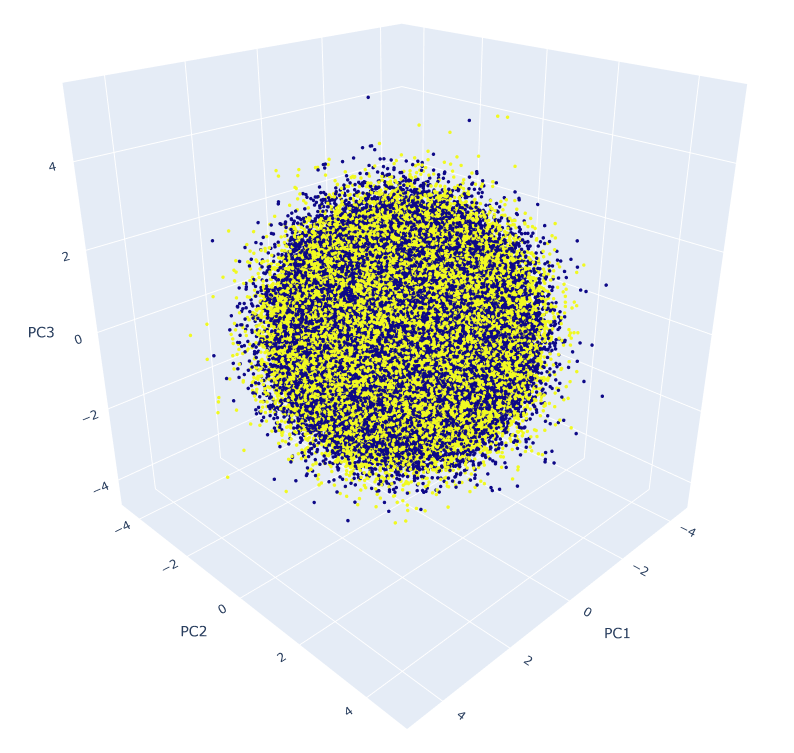}
		\end{adjustwidth}
	\end{figure}
	
	In the case of class number 1 and 3 fields in Figures \ref{fig:3d-PCA-all-data-h=1vs3}, \ref{fig:PCA-ap-h=1vs3} and \ref{fig:PCA-3d_disc+reg+ai(1)_h=1vs3}, we see that the data is much less separable from the point of view of PCA. In contrast to Figure \ref{fig:3d-PCA-all-data-h=1vs2}, Figure \ref{fig:3d-PCA-all-data-h=1vs3} shows that the coefficients of class number 1 and 3 fields are completely mixed, which can explain the greater difficulty in classifying such fields, and there is no noticeable difference between Figure \ref{fig:PCA-ap-h=1vs3} and Figure \ref{fig:3d_PCA(3)-ap-1}. Figure \ref{fig:PCA-3d_disc+reg+ai(1)_h=1vs3} shows the result when the first 10 prime coefficients are combined with other features $R_d$ and $D$. Recall that, with all these features combined, we could obtain a high accuracy as indicated in Table \ref{tab:class_number_1vs3}. 
	
	\begin{figure}[ht!]
		\caption{\sf Dimensionality reduction using PCA for $(a_n)_{n=1}^{1000}$ in the case of real quadratic fields of class number $h_d \in \{1,3 \}$  (1: blue;  3: yellow) }
		\label{fig:3d-PCA-all-data-h=1vs3}
		\begin{adjustwidth}{-.5in}{-.5in} 
			\centering
			\subfloat{\includegraphics[width=8cm]{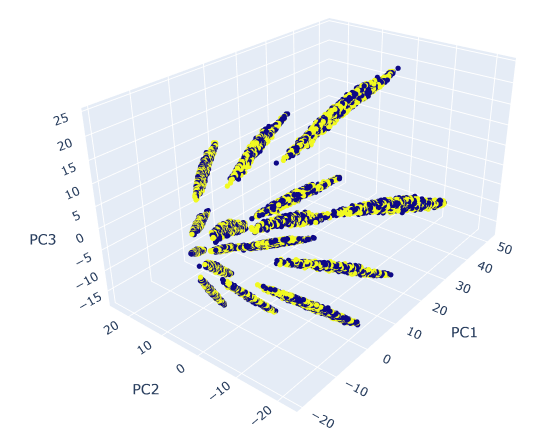}}
		\end{adjustwidth}
	\end{figure}
	
	\begin{figure}
		\caption{\sf Dimensionality reduction using PCA for $(a_p)_{p\leq 1000 \atop p: \text{prime} }$ in the case of real quadratic fields of class number $h_d \in \{ 1,3 \}$ (1:blue; 3: yellow) }
		\label{fig:PCA-ap-h=1vs3}
		\begin{adjustwidth}{-.5in}{-.5in} 
			\centering
			\includegraphics[width=8cm]{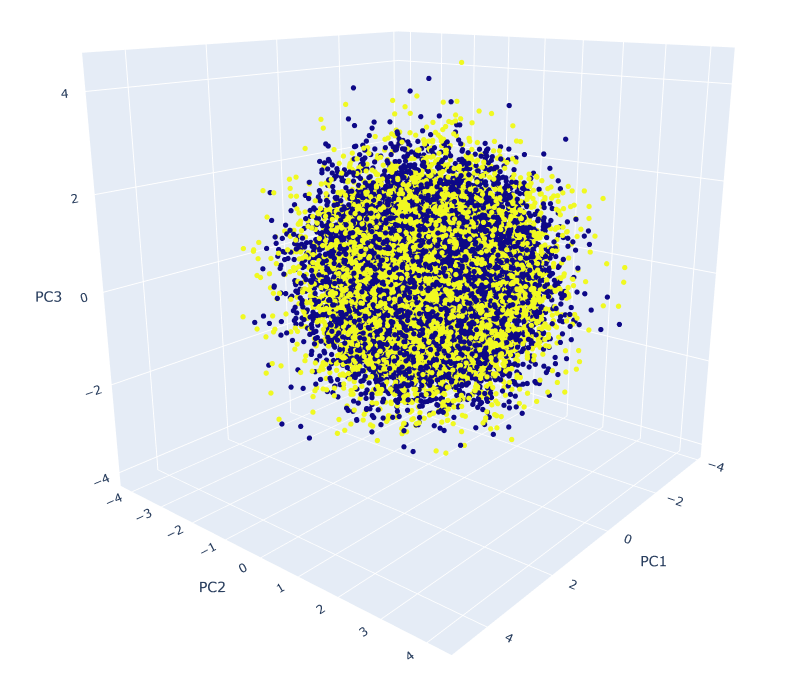}
		\end{adjustwidth}
	\end{figure}
	
	\begin{figure}
		\caption{\sf Dimensionality reduction using PCA for $D,R_d$ and the first 10 $a_p$-coefficients in the case of real quadratic fields of class number $h_d \in \{1,3\}$ (1:blue; 3: yellow)}
		\label{fig:PCA-3d_disc+reg+ai(1)_h=1vs3}
		\begin{adjustwidth}{-.5in}{-.5in} 
			\centering
			\subfloat{\includegraphics[width=8cm]{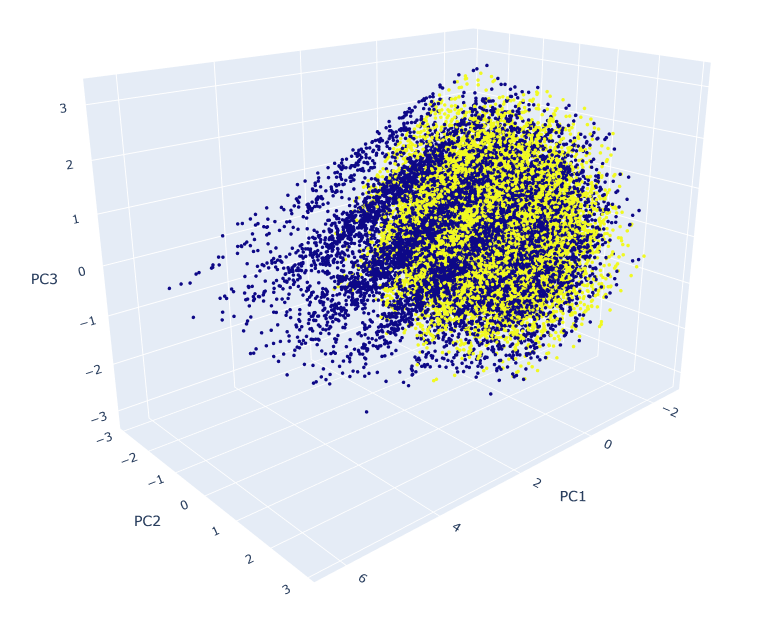}}
		\end{adjustwidth}
	\end{figure}

	\subsection{Supplementary material for Section~\ref{ss:pi}}\label{s:app}

	In Section~\ref{ss:pi}, we focused largely on the permutation feature importance. 
	In Figures~\ref{fig:calibration-LightGBM} and ~\ref{fig:confusion-LightGBM}, we record the 
	calibration curves and  confusion matrices for 70/30 splits.
	
	\begin{figure}[ht!]
		\caption{\sf Calibration curves of the LightGBM model (left) and CatBoost model (right) on a 70/30 split of the dataset for the binary classification task of quadratic fields with class number $h_d=1$ vs $h_d=2$.}
		\label{fig:calibration-LightGBM}
		\begin{adjustwidth}{-.5in}{-.5in} 
			\centering
			\subfloat{{\includegraphics[width=9cm]{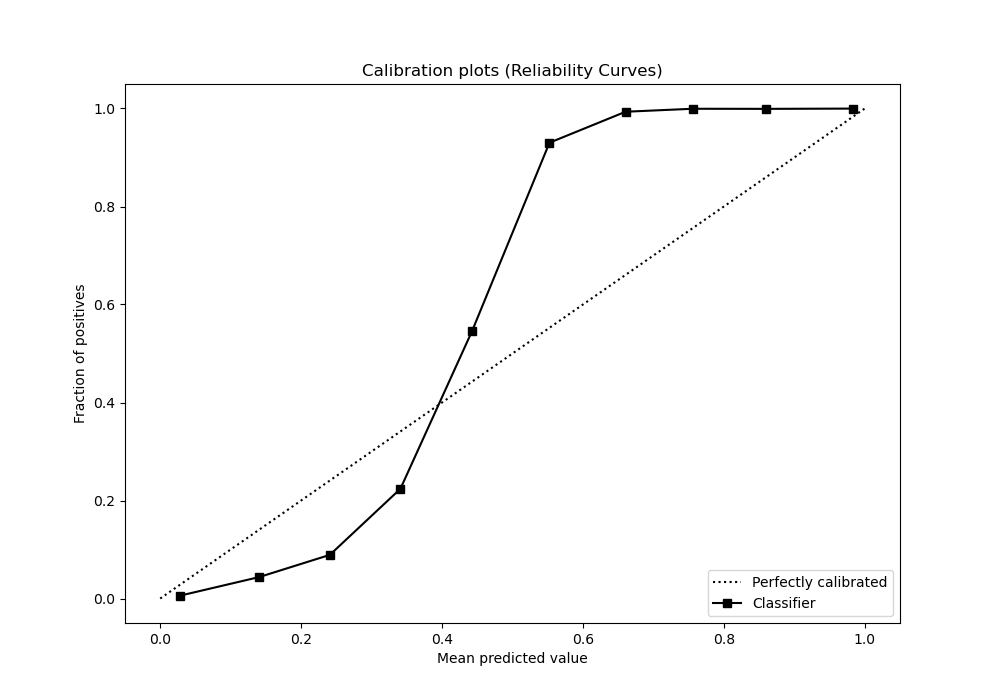} }}
			\subfloat{{\includegraphics[width=9cm]{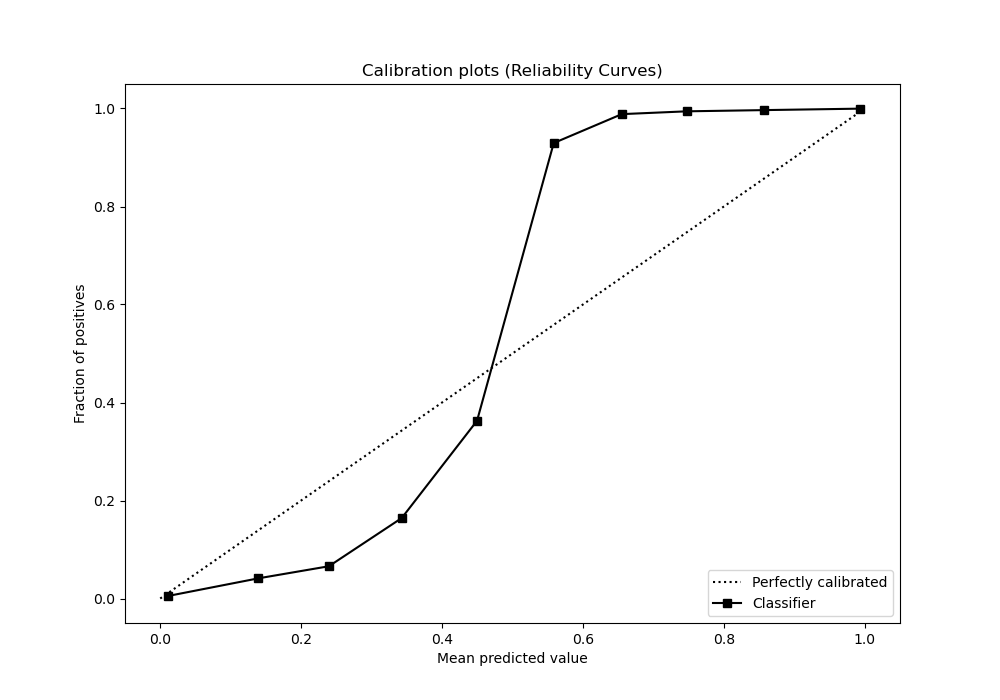} }}
		\end{adjustwidth}
	\end{figure}
	
	\begin{figure}[ht!]
		\caption{\sf Confusion matrix of the LightGBM model (left) and CatBoost model (right) on a 70/30 split of the dataset for the binary classification task of quadratic fields with class number $h_d=1$ vs $h_d=2$.}
		\label{fig:confusion-LightGBM}
		\begin{adjustwidth}{-.5in}{-.5in} 
			\centering
			\subfloat{{\includegraphics[width=9cm]{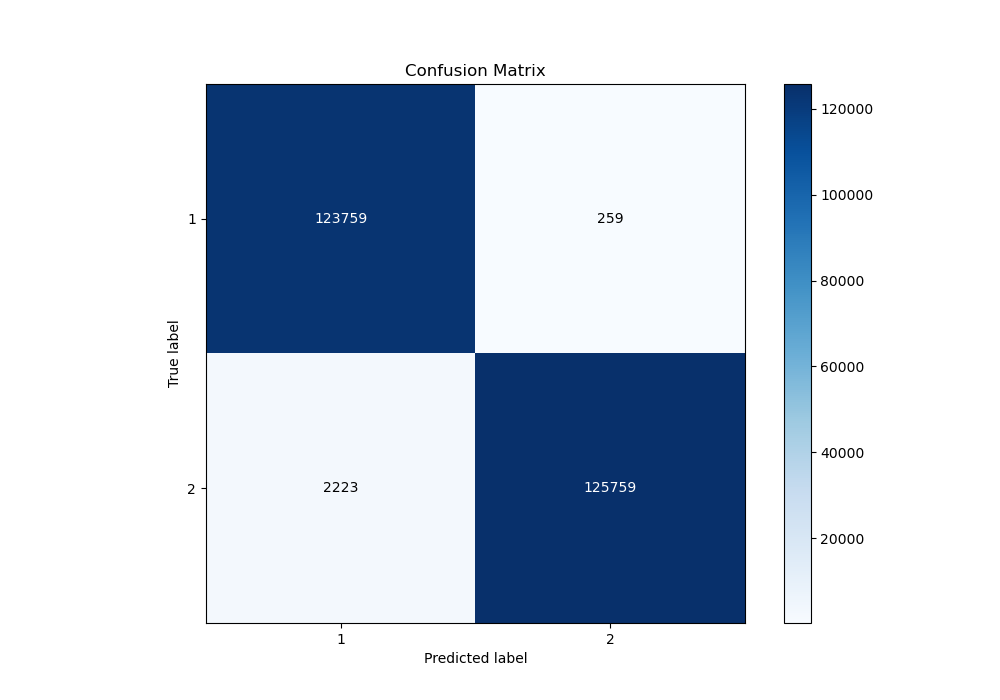} }}
			\subfloat{{\includegraphics[width=9cm]{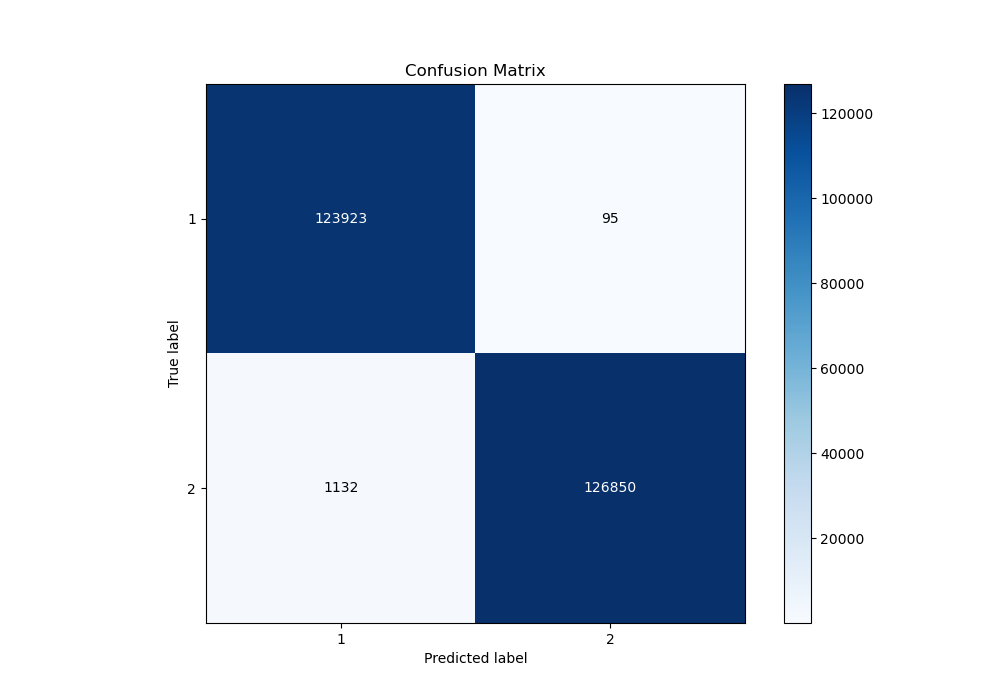} }}
		\end{adjustwidth}
	\end{figure}

\end{document}